%% file: Weyl-main-revision-also-for-arxiv.tex
\documentclass[11pt,a4paper]{article}

\usepackage[utf8]{inputenc}					% for unicode support

\usepackage{lmodern}						% font
\usepackage[english]{babel}					% babel

\usepackage{amsmath,amsfonts,amssymb,amsthm}
\allowdisplaybreaks

\usepackage{stmaryrd}

\usepackage{todonotes}
\usepackage{tikz-cd}

\usepackage[left=3cm,right=3cm,top=3cm,bottom=3cm]{geometry}

\usepackage{mathtools}						% uncomment to tag only the referred equations (post-production)
\mathtoolsset{showonlyrefs,showmanualtags}	% uncomment to tag only the referred equations (post-production)

\usepackage{enumitem}

\DeclareMathAlphabet{\mathbbold}{U}{bbold}{m}{n}	% nuovo alfabeto per i numeri in bold.

\theoremstyle{plain}
\newtheorem{thm-intro}{Theorem}
\newtheorem{cor-intro}[thm-intro]{Corollary}

\newtheorem*{theorem*}{Theorem}
\newtheorem{theorem}{Theorem}[section]
\newtheorem{prop}[theorem]{Proposition}
\newtheorem{cor}[theorem]{Corollary}
\newtheorem{corollary}[theorem]{Corollary}
\newtheorem{lemma}[theorem]{Lemma}

\theoremstyle{definition}
\newtheorem{definition}[theorem]{Definition}

\newtheorem*{assumption*}{\assumptionnumber}
\providecommand{\assumptionnumber}{}
\makeatletter
\newenvironment{assumption}[1]
 {%
  \renewcommand{\assumptionnumber}{Assumption #1}%
  \begin{assumption*}%
  \protected@edef\@currentlabel{#1}%
 }
 {%
  \end{assumption*}
 }
\makeatother

\theoremstyle{remark}
\newtheorem{rmk}{Remark}[section]
\newtheorem{example}{Example}[section]

\DeclareMathOperator{\tr}{\mathrm{Tr}}
		
\DeclareMathOperator{\supp}{supp}

\DeclareMathOperator{\Hess}{Hess}

\DeclareMathOperator{\Sec}{\mathrm{Sec}}
\DeclareMathOperator{\Ric}{Ric}
\DeclareMathOperator{\inj}{inj}

\newcommand{\N}{\mathbb{N}}					% natural numbers
\newcommand{\R}{\mathbb{R}}					% real numbers
\renewcommand{\epsilon}{\varepsilon}		% varepsilon
\newcommand{\M}{\mathbb{M}}
\newcommand{\SSS}{\Sigma}				% density one subset
\newcommand{\KK}{M}						% compact manifold

\newcommand{\vol}{\mathrm{vol}}

\newcommand{\dom}{D}

\newcommand{\distb}{d_{\partial}}
\newcommand{\distmb}{\delta}
\newcommand{\vv}{\upsilon}

\usepackage{hyperref}
\hypersetup{
    colorlinks,
    linkcolor={red!50!black},
    citecolor={blue!50!black},
    urlcolor={blue!80!black},
}

\title{Weyl's law for singular Riemannian manifolds}

\author{Y.\ Chitour\thanks{Laboratoire des Signaux et Systèmes, Universit\'e Paris-Sud, CentraleSup\'elec, Gif-sur-Yvette},
D.\ Prandi\thanks{CNRS, Laboratoire des Signaux et Systèmes, CentraleSup\'elec, Gif-sur-Yvette}, and
L.\ Rizzi\thanks{SISSA, Trieste, Italy \& Univ. Grenoble Alpes, CNRS, Institut Fourier, F-38000 Grenoble, France}}

\begin{document}

\maketitle

\begin{abstract}
We study the asymptotic growth of the eigenvalues of the Laplace-Beltrami operator on singular Riemannian manifolds, where all geometrical invariants appearing in classical spectral asymptotics are unbounded, and the total volume can be infinite. 
Under suitable assumptions on the curvature blow-up, we show how the singularity influences the Weyl's asymptotics. Our main motivation comes from the construction of singular Riemannian metrics with prescribed non-classical Weyl's law. Namely, for any non-decreasing slowly varying function $\vv$ we construct a singular Riemannian structure whose spectrum is discrete and satisfies
\[
N(\lambda) \sim \frac{\omega_n}{(2\pi)^n} \lambda^{n/2} \vv(\lambda).
\]
Examples of slowly varying functions are $\log\lambda$, its iterations $\log_k \lambda = \log_{k-1}\log\lambda$, any rational function with positive coefficients of $\log_k \lambda$, and functions with non-lo\-ga\-rithmic growth such as $\exp\left((\log \lambda)^{\alpha_1} \dots (\log_k \lambda)^{\alpha_k} \right)$ for $\alpha_i \in (0,1)$.
A key tool in our arguments is a new quantitative estimate for the remainder of the heat trace and the Weyl's function on Riemannian manifolds, which is of independent interest.
\end{abstract}

\setcounter{tocdepth}{2}
\tableofcontents

\input{introduction.tex}

\input{heatestimates.tex}

\input{weyllaw.tex}

\input{applications.tex}

\appendix

\input{geometricbounds.tex}

{\small

\bibliographystyle{abbrv}
\bibliography{biblio}

}

\end{document}

%% file: introduction.tex
\section{Introduction}\label{s:intro}
 
In spectral geometry,  the asymptotic behaviour of the eigenvalues of the Laplace–Beltrami operator on a smooth and compact Riemannian manifold $\M$ is given by the Weyl's law:
\begin{equation}\label{eq:wl1}
  \lim_{\lambda\to \infty} \frac{N(\lambda)}{\lambda^{n/2}} = \frac{\omega_n}{(2\pi)^n} \vol(\M).
\end{equation}
Here, $N(\lambda)$ is the number of eigenvalues smaller than $\lambda$ for the Dirichlet Laplacian on a smooth Riemannian manifold $\M$, possibly with boundary, $\vol(\M)$ stands for the Riemannian volume of $\M$ and $\omega_n$ is the volume of the $n$-dimensional Euclidean unit ball.

The study of eigenvalue asymptotics is a rich topic with a long history. The subject has been treated in several setting and with different methods, see for example \cite{Ivrii100,Ivrii-II,Shubin,Chavel}. In this paper, we focus on Weyl's-type asymptotics for the Laplace-Beltrami operator of a class of singular Riemannian structures, where all geometric invariants, including the curvature and the volume, can be unbounded when approaching the singularity. %{\red In particular, we aim at understanding to which extent the presence of the singularity modifies the Weyl's asymptotics \eqref{eq:wl1}.}

\subsection{The Grushin sphere model}\label{s:grushin}

We first discuss a simple model, emphasizing the peculiarities of our setting and the differences with existing results. Consider the two-dimensional sphere $\mathbb{S}^2\subset \R^3$. Let $X$ and $Y$ be the generators of rotations around the $x$ and $y$ axis, respectively. We define a Riemannian structure by declaring $X$ and $Y$ to be orthonormal. These vector fields are collinear on the equator $S = \{(x,y,z) \in \mathbb{S}^2\mid z=0 \}$, and hence the metric tensor we defined is singular on $S$ (the coefficients of the metric explode). This is an almost-Riemannian structure in the sense of \cite{ABS-GaussBonnet,BL-ARS}. In cylindrical coordinates $(\theta,z) \in (0,2\pi)\times (-1,1)$, the associated Laplace-Beltrami operator $\Delta$, with domain $C^\infty_c(\mathbb{S}^2\setminus S)$ is 
\begin{equation}\label{eq:LB-Grushin}
-\Delta =\frac{\partial^2}{\partial z^2}+z^2\frac{\partial^2}{\partial \theta^2} +\left(\frac{1}{z}-z\right) \frac{\partial}{\partial z}.
\end{equation}
By construction, $\Delta$ is symmetric on $L^2(\mathbb{S}^2\setminus S,d\mu_g)$, where the Riemannian measure is
\begin{equation}
d\mu_g = \frac{1}{|z|}d\theta dz.
\end{equation}
It turns out that $\Delta$ is essentially self-adjoint with compact resolvent \cite{BL-ARS}. The spectrum can be computed, cf.\ \cite{BPS-Bohm}, and it satisfies the following non-classical Weyl's asymptotics:
\begin{equation}
N(\lambda) \sim  \frac{1}{4} \lambda\log\lambda, \qquad \lambda \to \infty.
\end{equation}
Despite the problem taking place on a relatively compact space, the total Riemannian volume is infinite and the curvature explodes to $-\infty$ when approaching the equator. Hence, on-diagonal small-time heat kernel estimates blow up at the singular region. In particular, it is not clear how to deduce the asymptotics of $N(\lambda)$ using classical Tauberian techniques.

\medskip
The class of singular structures that we study in this paper is inspired by the Grushin sphere, and it is determined by the control on the blow-up of intrinsic quantities such as curvature, injectivity radius, cf.\ Assumption \ref{a:singularity}. This class of geometric structures is particularly relevant in questions arising in sub-Riemmanian geometry. The corresponding Laplace-Beltrami operators should not be confused with the class of degenerate operators studied in the vast literature in microlocal analysis, cf.\ for example \cite{Sjostrand-IV,BCP-EDP} and references within, or \cite[Ch.\ 12]{Ivrii-II}. In these references, general operators with degenerate symbol are studied on suitable $L^2$ spaces, with respect to a measure that remains regular at the degeneration locus. In all these cases the corresponding integral kernels remain well-defined at the degeneration locus, and their asymptotics there can be recovered via microlocal methods (cf.\ for example \cite[Thm.\ 12.5.10]{Ivrii-II}). Our class of operators is not even defined at the singularity, and the full symbol has singular terms, as it is clear from example \eqref{eq:LB-Grushin}.

\subsection{Setting and main results}

Let $(\M,g)$ be a non-complete Riemannian manifold. Intrinsic quantities such as the curvature, the measure of balls, et cætera, can blow up when approaching the metric boundary of $\M$, which we thus consider as a singularity. We require the following assumption.

\begin{assumption}{$\mathbf{A}$}\label{a:singularity}
Let $\delta$ be the distance from the metric boundary of $\M$. Then, there exists a neighborhood $U = \{\delta < \varepsilon_0\}$ on which the following hold:
\begin{itemize}
\item[(a)] regularity: $\delta$ is smooth;
\item[(b)] convexity: the level sets of $\delta$ are convex, i.e., $\Hess(\delta) \leq 0$;
\item[(c)] curvature control: there exists $C>0$ such that $|\Sec | \leq C \delta^{-2}$;
\item[(d)] injectivity radius control: there exists $C>0$ such that $\inj \geq C \delta$.
\end{itemize}
\end{assumption}

By (a), we identify $U \simeq (0,\varepsilon_0)\times Z$, for a fixed $(n-1)$-dimensional manifold $Z$ without boundary. The metric on $U$ has the form
\begin{equation}\label{eq:metricnormalform}
g = dx^2 + h(x),
\end{equation}
where $h(x)$ is a smooth family of Riemannian metrics on $Z$, and $\delta(x,z) = x$ for $(x,z)\in U$. %{\red The convexity assumption (b) implies that, for any $V\in TZ$, the map $x\mapsto h(x)(V,V)$ is non-increasing. The remaining assumptions impose additional constraints on $h(x)$.}

%{\red
%\begin{rmk}
%Assumption \ref{a:singularity} implies that the singularity can be approached through the exhaustion $\M_\varepsilon^\infty = \{\varepsilon\leq  \delta \leq \infty \}$ of smooth Riemannian manifold with convex boundary, with curvature explosion at most quadratic and injectivity radius going to zero at most linearly as $\varepsilon\to 0$, which is all what is needed in the proofs.
%\end{rmk}
%}
\begin{rmk}\label{r:improvements}
Item (d) is implied by the others if the convexity is strict, or if the metric is of warped product type in a neighborhood of the singularity, cf.\ Proposition~\ref{prop:inj-strict-convex}. We do not known whether (d) is independent from the other assumptions in full generality.
\end{rmk}
\begin{rmk}
Assumption (b) implies that the sectional curvature cannot explode to $+\infty$ in the following sense: for any lower bound $K_\varepsilon$ of $\Sec$ on $\M_\varepsilon^\infty=\{\varepsilon\le \delta\le \infty\}$, one has $\liminf_{\varepsilon \to 0} K_\varepsilon < +\infty$. However, one can build examples satisfying Assumption \ref{a:singularity} with curvature oscillating between $\pm \infty$ as $\delta \to 0$. (E.g., take $\vv(\lambda) = 3\log\lambda + \sin\log\lambda$ in the construction of the proof of Theorem \ref{t:converse}.)
\end{rmk}
%\begin{rmk}
%\blue We explain how to remove the smoothness assumption (a) in Section \ref{s:removesmoothness}. This generalization does not bring to immediate new applications or insights, so we prefer to work under the regularity assumption for the rest of the paper.
%\end{rmk}

Let $\Delta$ be the Friedrichs extension of the Laplace-Beltrami operator on $(\M,g)$, that is the unique self-adjoint operator in $L^2(\M,d\mu_g)$ associated with the quadratic form
\begin{equation}
Q(u) = \int_{\M} |\nabla u|^2 d\mu_g, \qquad \forall u \in C^\infty_c(\M).
\end{equation}

To quantify the rate of growth of the volume at the singularity, let $\M_{\varepsilon}^\infty$ be the set at distance greater than $\varepsilon>0$ from the metric boundary, and define
\begin{equation}
\vv(\lambda):= \vol\left(\M_{1/\sqrt{\lambda}}^\infty\right).
\end{equation}
% {\red Our first} {\blue One preliminary} result, proved in Section \ref{s:weyl}, is the following.
% \begin{theorem}\label{t:weyl-intro}
%   Let $\M$ be a $n$-dimensional Riemannian manifold with compact metric completion and satisfying Assumption~\ref{a:singularity}. Then, there exist $C_{\pm}>0$ and $\Lambda>0$ such that
%   \begin{equation}\label{eq:weylboundsintro}
%     C_-\le \frac{N(\lambda)}{\lambda^{n/2}\vv(\lambda)}\le C_+,\qquad \forall \lambda \geq \Lambda.
%   \end{equation}
% %Here, $\M_{\varepsilon}^\infty$ is the set at distance greater than $\varepsilon>0$ from the metric boundary.
% \end{theorem}
% It is not clear whether the limit for $\lambda \to \infty$ exists in our general setting.
Our main result is a precise Weyl's law under an additional assumption on the volume growth, ruling out rapid oscillations and growth. See Theorem \ref{t:weylexact}. 
\begin{theorem}\label{t:weylexact-intro}
Let $\M$ be an $n$-dimensional Riemannian manifold with compact metric completion and satisfying Assumption~\ref{a:singularity}. Then, if $\vv$ is slowly varying, we have
\begin{equation}\label{eq:weyl-exact-intro}
        \lim_{\lambda\to \infty} \frac{N(\lambda)}{\lambda^{n/2}\vv(\lambda)} = \frac{\omega_n}{(2\pi)^n}.
    \end{equation}
    Recall that $\vv$ is slowly varying at infinity if $\vv(a\lambda)\sim \vv(\lambda)$ as $\lambda \to \infty$ for all positive $a$.
\end{theorem}
\begin{rmk}
Examples of slowly varying functions are logarithms and their iterations
\begin{equation}
\log\lambda, \qquad \log_k \lambda= \log_{k-1}\log \lambda, \qquad k =2,3,\dots,
\end{equation}
and any rational function with positive coefficients formed with the above. This class also contains functions with non-logarithmic growth such as
\begin{equation}
\exp\left((\log \lambda)^{\alpha_1} \dots (\log_k \lambda)^{\alpha_k} \right), \qquad 0<\alpha_i<1.
\end{equation}
\end{rmk}
\begin{rmk}
 The assumptions of Theorem \ref{t:weylexact-intro} are verified for the Grushin sphere of Section \ref{s:grushin}, and more generally for generic $2$-dimensional ARS without tangency points \cite{BL-ARS}. In these cases, $\vv(\lambda) = \sigma \log\lambda$ for some $\sigma>0$ depending on the structure, see Section \ref{s:ARS}.
\end{rmk}

We now turn to the inverse problem of building structures with prescribed large eigenvalues asymptotic. Our next main result can be seen as a counterpart at infinity of a celebrated result of Colin de Verdière \cite{Yves-partofspectrum} stating that, for any finite sequence of numbers $0 < \lambda_1 \leq \lambda_2 \leq \dots \leq \lambda_m$, one can find a compact Riemannian manifold such that these numbers are the first $m$ eigenvalues. See also \cite{Lohkamp}. See Theorem \ref{t:converse}.

\begin{theorem}\label{t:converse-intro}
Let $\KK$ be an $n$-dimensional compact manifold, $S\subset \KK$ be a closed submanifold, and $\vv : \R_+ \to \R_+$ be a non-decreasing slowly varying function. Then, there exists a Riemannian structure on $\KK$, singular at $S$, such that Weyl's law \eqref{eq:weyl-exact-intro} holds for the non-complete Riemannian manifold $\M = \KK\setminus S$.
\end{theorem}

We stress that as a consequence of the construction in the proof of Theorem \ref{t:converse-intro}, the corresponding Laplace-Beltrami operator is essentially self-adjoint, see Remark~\ref{r:ess-adj}.

\medskip

  Without the slow variation assumption on $\vv$, we are unable to prove an exact Weyl law. It is indeed not clear that such a law exists in this general setting. We have however the following weaker version of Theorem~\ref{t:weylexact-intro}. See Theorem \ref{t:weyl}.

 \begin{theorem}\label{t:weyl-intro}
   Let $\M$ be a $n$-dimensional Riemannian manifold with compact metric completion and satisfying Assumption~\ref{a:singularity}. Then, there exist $C_{\pm}>0$ and $\Lambda>0$ such that
   \begin{equation}\label{eq:weylboundsintro}
     C_-\le \frac{N(\lambda)}{\lambda^{n/2}\vv(\lambda)}\le C_+,\qquad \forall \lambda \geq \Lambda.
   \end{equation}
 \end{theorem}

Via a classical argument, we also prove the concentration of eigenfunctions at the metric boundary, in presence of a non-classical Weyl's asymptotics. It applies in particular to all structures satisfying Assumption~\ref{a:singularity}, with compact metric completion, and with infinite volume.  See Theorem \ref{t:localization}.

\begin{theorem}\label{t:localization-intro}
Let $\M$ be an $n$-dimensional Riemannian manifold such that the Laplace-Beltrami operator $\Delta$ has discrete spectrum, and
\begin{equation}\label{eq:noncompact}
\lim_{\lambda \to \infty} \frac{N(\lambda)}{\lambda^{n/2}} = \infty.
\end{equation}
Let $\{\phi_i\}_{i\in \N}$, be a complete set of normalized eigenfunctions of $-\Delta$, associated with eigenvalues $\lambda_i$, arranged in non-decreasing order. Then, there exists a density one subset $\SSS \subseteq \N$ such that for any compact $U$ it holds
\begin{equation}
\lim_{\substack{i \to \infty \\ i \in \SSS}} \int_U |\phi_i|^2d\mu_g = 0.
\end{equation}
\end{theorem}

\subsection{Quantitative remainders for heat trace and counting function}

In order to highlight a key technical result of independent interest, we sketch here the proof of Theorem~\ref{t:weyl-intro}. It consists in the simultaneous exploitation of Dirichlet-Neumann bracketing (classically used in the Euclidean case) and Tauberian techniques (classically used when all intrinsic geometric quantities are bounded).

The idea is to consider the splitting $\M = \M_0^{\varepsilon}\cup\M_\varepsilon^\infty$ in a boundary (singular) part and an inner (regular) one. By Dirichlet-Neumann bracketing, $N(\lambda)$ is controlled by the counting functions for the Laplace-Beltrami operator on the two domains, with Neumann $(+)$ or Dirichlet $(-)$ boundary conditions, respectively:
\begin{equation}\label{eq:idea1}
N^-_{[0,\varepsilon]}(\lambda)+N^-_{[\varepsilon,\infty]}(\lambda) \leq N(\lambda)\leq N^+_{[0,\varepsilon]}(\lambda)+N^+_{[\varepsilon,\infty]}(\lambda).
\end{equation}

Thanks to the convexity assumption, $\M_0^\varepsilon$ supports a Hardy-type inequality. As a consequence, $N^{\pm}_{[0,\varepsilon(\lambda)]}(\lambda) = 0$, provided that $\varepsilon =\varepsilon(\lambda) \to 0$ sufficiently fast (in a quantitative way) as $\lambda \to \infty$. In this regime, the asymptotics of $N(\lambda)$ is controlled by the Weyl function of the truncation $\M_{\varepsilon(\lambda)}^\infty$. The latter is a Riemannian manifold with boundary and finite volume, which satisfies indeed the classical Weyl's law
\begin{equation}\label{eq:idea2}
N_{[\varepsilon(\lambda),\infty]}^{\pm}(\lambda) \sim \frac{\omega_n}{(2\pi)^n} \vol\left(\M_{\varepsilon(\lambda)}^\infty\right) \lambda^{n/2}.
\end{equation}
The implicit remainder in \eqref{eq:idea2}, which depends on the parameter of the truncation $\varepsilon(\lambda)$, must be carefully controlled as $\lambda \to \infty$.  The key is the following heat-trace asymptotic formula with remainder for compact Riemannian manifolds with convex boundary. See Theorem \ref{t:traceheat} and Corollary \ref{c:traceheat}.
%{\red
%\begin{theorem}\label{t:traceheat-intro}
%Let $(M,g)$ be a compact $n$-dimensional Riemannian manifold with convex boundary $\partial M$. Let $K,H\geq 0$ such that $|\Sec(M)|\le K$ and $|\Hess(\distb)| \leq H$ for $\distb < \inj_{\partial}(M){\blue /2}$.
%Then there exists a constant $c>0$, depending only on $n$, such that the following estimate for the Dirichlet and Neumann heat kernels $E^{\pm}$ holds:
%\begin{equation}
%\left|\frac{(4\pi t)^{n/2}}{\vol(M)\phantom{}}\int_M E^{\pm}(t,q,q) d\mu_g(q) - 1\right|  \leq  c \left(\frac{t}{t_0}\right)^{1/2},
%%c_1 Kt + c_2 \frac{\vol(\partial M)}{\vol(M)} \sqrt{t} + c_3 e^{- c_4 \tfrac{\mathfrak{i}^2}{4t}},
%\end{equation}
%for all values of $t \in \R_+$ such that
%\begin{equation}
%\sqrt{t}\leq \sqrt{t_0}= \min\left\{\inj(M),\frac{\inj_{\partial}(M)}{{\red 2}{\blue 4}}, \frac{\pi}{\sqrt{K}}, \frac{1}{H}\right\}.
%\end{equation}
%\end{theorem}
%}

\begin{theorem}\label{t:traceheat-intro}
Let $(M,g)$ be a compact $n$-dimensional Riemannian manifold with convex boundary $\partial M$. Let $K\geq 0$ such that $|\Sec(M)|\le K$.
Then there exists a constant $c>0$, depending only on $n$, such that the following estimate for the Dirichlet or Neumann heat kernels $E^{\pm}$ holds:
\begin{equation}
\left|\frac{(4\pi t)^{n/2}}{\vol(M)\phantom{}}\int_M E^{\pm}(t,q,q) d\mu_g(q) - 1\right|  \leq  c \left[ (t/t_0)^{1/2} + (t/t_0)^{n/2}\right] ,
\end{equation}
for all $t \in \R_+$, and where $\sqrt{t_0}:= \min\left\{\inj(M),\frac{\inj_{\partial}(M)}{4}, \frac{\pi}{\sqrt{K}}\right\}$.
\end{theorem}

  \begin{rmk}
    Here, $\distb$ and $\inj_{\partial}(M)$ are the distance and the injectivity radius from $\partial M$, respectively. Furthermore, we used $M$ to denote a compact Riemannian manifold with boundary in order to distinguish it from the non-complete manifold $\M$.
  \end{rmk}

\begin{rmk}
Some assumptions in Theorem \ref{t:traceheat-intro} can be dispensed of if one is interested only in upper or lower bounds, or only to Dirichlet or Neumann boundary conditions. For example, an upper bound for the Dirichlet heat trace does not require convexity and, in this case, the remainder in Theorem \ref{t:traceheat-intro} can be replaced by a lower bound on $\Ric$ and an upper bound on $\Sec$. We adopted an unified approach with non-optimal assumptions to include all cases in a single simpler statement.
\end{rmk}
\begin{rmk}
  The remainder of Theorem \ref{t:traceheat-intro} is fully explicit in terms of the geometry of $M$, as well as intrinsic, in the sense that it does not rely on an embedding of $M$ in a larger manifold. Compare for instance to the similar remainder formula in \cite[Thm.\ 3.5]{Donnelly-Lee-remainders} where the remainder is controlled by the geometry of a manifold $M'$ in which $M$ is embedded. Although we will indeed apply Theorem~\ref{t:traceheat-intro} to compact manifolds that are naturally embedded in a larger one, the latter will have unbounded curvature and injectivity radius and thus \cite[Thm.\ 3.5]{Donnelly-Lee-remainders} does not apply. We also remark that the remainder in \cite[Thm.\ 3.5]{Donnelly-Lee-remainders} contains terms that depend in a non-explicit way on the geometry.
\end{rmk}
As a consequence of Theorem \ref{t:traceheat-intro}, and a suitable Karamata-type theorem with remainder, we obtain an asymptotic formula with universal remainder\footnote{By ``universal remainder'', we mean that it depends on the structure only through a handful of geometrical invariants, and fixed dimensional constants.} for the eigenvalue counting function of the Laplace-Beltrami operator on a compact Riemannian manifold with convex boundary. See Theorem \ref{t:weylremainder}.

\begin{theorem}\label{t:weylremainder-intro}
 Let $(M,g)$ be a compact $n$-dimensional Riemannian manifold with convex boundary $\partial M$. Let $K \geq 0$ such that $|\Sec(M)|\le K$. Then, there exists a constant $C>0$, depending only on $n$, such that the following estimate holds for the counting function for Dirichlet or Neumann eigenvalues:
\begin{equation}
\left\lvert\frac{N(\lambda)}{\tfrac{\omega_n}{(2\pi)^n}\vol(M)\lambda^{n/2}}-1\right\rvert \leq \frac{C}{\log(1+\sqrt{\lambda/\lambda_0})}, \qquad \forall \lambda >0,
\end{equation}
with $\sqrt{\lambda_0} =  \min\left\{\inj(M),\frac{\inj_\partial(M)}{4}, \frac{\pi}{\sqrt{K}}\right\}^{-1}$.
\end{theorem}

When applied to $\M_{\varepsilon(\lambda)}^\infty$, Theorem \ref{t:weylremainder-intro}  singles out the quantities whose explosion must be controlled as $\lambda \to \infty$ and $\varepsilon \to 0$ in \eqref{eq:idea2} concluding the proof of Theorem \ref{t:weyl-intro}. On the other hand, the proof of Theorem \ref{t:weylexact-intro} is much more delicate:  one must consider a three-parts splitting, with an additional intermediate buffer region, whose contribution must be shown to be negligible compared to the one of the inner region.

\begin{rmk}\label{rmk:sharpness}
 An important feature of Theorem \ref{t:weylremainder-intro} is its asymptotic sharpness as $\lambda \to \infty$, in the sense that in the limit it yields the exact classical Weyl's law for compact manifolds with boundary. This feature, key in the proof of Theorem \ref{t:weylexact-intro}, should be seen as one of the main novelty with respect to other similar estimates in the literature (usually found in the equivalent form of estimates for eigenvalues, see \cite[Thm.\ 1.2.6 and Thm.\ 1.3.8]{Hassannezhad}). The price to pay for sharpness is that our formulas rely on stronger curvature bounds.
\end{rmk}
\begin{rmk}\label{rmk:smoothness}
The smoothness hypothesis (a) in Assumption \ref{a:singularity} is technical and we are unsure whether it can be removed or not for Theorem \ref{t:weylexact-intro}: in this case, the convexity hypothesis (b) should be understood in the sense of distributions; the Hardy inequality (Proposition \ref{lem:hardy}) still holds in this case but it is not clear how to extend the validity of a few key technical results, and in particular Proposition \ref{p:intermediatecouche} for the buffer region.
\end{rmk}

 %The second main merit is, indeed, the fact that the remainder estimate is given in terms of the explicit intrinsic geometry of $M$, and universal constants.%(see also \cite{ColboisMaerten-Neumann}). 

%In the following, we formulate the upper bound of Theorem \ref{t:weylremainder} in a weaker and non-asymptotically sharp form, similar to Buser's inequality.
%
%
%\begin{corollary}\label{c:Buser}
%In the setting of Theorem \ref{t:weylremainder-intro}, there exists a constant $C>0$ depending only on $n$, such that the following estimate holds for the counting function for the Dirichlet or Neumann eigenvalues:
%\begin{equation}\label{eq:Buser}
%N(\lambda) \leq C \vol(M) \left( \lambda^{n/2} + \inj(M)^{-n} + \inj_\partial(M)^{-n} + K^{n/2}+H^{n}\right), \qquad \forall \lambda > 0.
%\end{equation}
%\end{corollary}
%Classical comparison estimates show that, under a lower curvature bound, $H$ can be estimated in terms of $K,\inj_\partial(M)$ and $\inj(M)$, so that the term containing $H$ in \eqref{eq:Buser} can be actually dispensed of, yielding a formula equivalent to the one in \cite[Sec.\ 1.2]{Hassannezhad}.
%

\subsection{Structure of the paper}
The first part of the paper, contained in Section~\ref{s:heat}, is devoted to the non-singular case. Here, for Riemannian manifolds with boundary, we prove remainder formulas for the heat trace asymptotics (Theorem \ref{t:traceheat-intro}), and for the Weyl's law (Theorem \ref{t:weylremainder-intro}).

In Section~\ref{s:singular} we present some preliminary results regarding singular Riemannian manifolds satisfying Assumption~\ref{a:singularity} which we exploit, in Section~\ref{s:weyl}, to prove Theorem~\ref{t:weyl-intro}. Section~\ref{s:exact} is dedicated to singular manifolds with slowly varying volumes, and in particular to the proofs of Theorem~\ref{t:weylexact-intro} and Theorem~\ref{t:converse-intro}, while in Section~\ref{sec:concentration} we prove Theorem~\ref{t:localization-intro} on the localization of eigenfunctions.
Finally, in Section \ref{s:ARS}, we apply our results to a class of almost-Riemannian structures, which generalize the Grushin sphere model.

We conclude the paper with two appendices  with technical estimates.

\subsection{Notation}
\label{sec:notation}

To avoid confusion, we employ two different notations for Riemannian manifolds:
\begin{itemize}
  \item $M$ for compact Riemannian manifolds, typically with boundary $\partial M$. This is used in particular throughout Section~\ref{s:heat} and Appendix~\ref{a:auxiliary}.
  \item $\M$ for typically non-complete Riemannian manifolds with compact metric completion, satisfying Assumption~\ref{a:singularity}. This is used in the rest of the paper.
\end{itemize}
% Naturally, all results concerning $\M$ reduce to the classical ones in the case where $\M$ is a compact manifold with possibly non-empty convex boundary.

\subsection{Other classes of singular structures}\label{s:literature}

There are several types of ``singular structures'' occurring in the literature. To put our contribution in perspective, we provide here a non-exhaustive overview.

\paragraph{Conical singularities.} There is a sharp difference between our class of singularities and conical ones \cite{C-conic}. In the latter case, our techniques do not apply since the boundaries of the truncations $\M_\varepsilon^\infty$ are concave as $\varepsilon \to 0$. However, the spectrum of the Laplace-Beltrami is still discrete, the total volume is finite, and the classical Weyl's law \eqref{eq:wl1} holds. In this sense, conical singularities are more gentle, and do not change the leading order of the counting function. Indeed, they are detected only at higher order, see \cite{S-conic1,S-conic2}.

\paragraph{Conformally cusp singularities.}
The spectral properties of conformally cusp type singularities have been studied in \cite{GoleniaMorianu}. In that reference, the authors derive a nice non-classical Weyl's law for the Laplace-Beltrami operator acting on $k$-forms, under suitable conditions on the topology of the singularity (which in particular exclude the case $k=0$). Note that the class of conformally cusp manifolds studied in \cite{GoleniaMorianu} does not contain our class of singularities: when the singularity is at finite distance (i.e., in the non-complete case) the structure is of metric horn type, cf.\ \cite{Cheeger-horn}. 

\paragraph{Structures with locally bounded geometry.}

In \cite{Lianantonakis2000}, the author considers non-complete Riemannian structures $(\M,g)$ equipped with a weighted measure $\sigma^2 d\mu_g$.
{The Riemannian measure $\mu_g$ and} the weight $\sigma$ might be singular at the metric boundary, and no regularity of the latter is assumed. 
In this setting, the author derives a rough Weyl's law similar to the one of Theorem \ref{t:weyl-intro}. The setting and methods of \cite{Lianantonakis2000} are rather different with respect to ours. The assumptions in \cite{Lianantonakis2000} imply that $\M$ is locally uniformly bi-Lipschitz equivalent to an Euclidean ball. In particular, if the metric completion is compact, their assumptions imply that the Riemannian volume of $\M$ is finite.

\paragraph{ARS with smooth measures.} An analogue to Theorem \ref{t:weylexact-intro} for $2$-dimensional ARS was announced in \cite{Y3} as a consequence of a more general local Weyl's law for sub-Laplacians \cite{Y1,Y2}. There, the authors are concerned with the Friedrichs extension associated with the quadratic form
\begin{equation}\label{eq:yves}
Q(u) = \int_{\KK} | \nabla u |^2 d\omega, \qquad u \in C^\infty(\KK),
\end{equation}
where $\KK$ is a smooth compact manifold carrying a smooth almost-Riemannian structure and the measure $\omega$ is positive and smooth on $\KK$, including on the singular region $S\subset \KK$. The reader not familiar with AR geometry can think at the example of the Grushin sphere discussed above, where $\KK = \mathbb{S}^2$ and the measure $\omega$ is the standard measure of the sphere. It is surprising that, for generic $2$-ARS, we obtain the same Weyl's law in our setting, where $\omega =\mu_g$ is singular on $S$ and the domain of the form \eqref{eq:yves} is $C^\infty_c(\KK\setminus S)$. See also \cite{Sjostrand-II} for a result covering in particular the Grushin sphere with smooth measure.

%{\red
%\paragraph{Metric measure spaces.} Recently, in \cite{AHT-Weyl,ZZ-Weyl}, the authors studied the convergence of heat kernels for sequences of $\mathrm{RCD}$ spaces (infinitesimally Hilbertian metric measure spaces with Ricci curvature bounded from below). This class includes measured Gromov-Hausdorff limits of complete Riemannian structures with Ricci curvature bounded from below and dimension bounded from above. As a consequence, the authors also prove that any $\mathrm{RCD}$ space satisfies the classical Weyl's law \eqref{eq:wl1}. Our contribution can also be seen as the first step toward the investigation of the Weyl's law for limits of Riemannian structures $(X_n,g_n,\mu_{g_n})$, where the Ricci curvature is unbounded.
%}
%{\red
%\paragraph{Rough metrics.} In \cite{rough}, the authors consider smooth compact manifolds with smooth boundary, equipped with a metric tensor which is only assumed to be bounded or measurable. This is known as a \emph{rough} Riemannian manifold. For a large class of boundary conditions the authors prove a classical Weyl's law for the asymptotics of the eigenvalues of the Laplacian associated to a rough metric. In this case the ``singularity'' is quite different with respect to the classes described so far, and is ascribed to a lack of differentiability of the metric.
%}

\subsubsection*{Acknowledgments}
This work was supported by the Grants ANR-15-CE40-0018, ANR-18-CE40-0012, and ANR-20-CE48-0003.
This work has been supported by a public grant overseen by the French National Research Agency (ANR) as part of the Investissement d'avenir program, through the iCODE project funded by the IDEX Paris-Saclay, ANR-11-IDEX-0003-02. L.R.\ has received funding from the European Research Council (ERC) under the European Union’s Horizon 2020 research and innovation programme (grant agreement No. 945655).
% {This project has received funding from the European Research Council (ERC) under the European Union’s Horizon 2020 research and innovation programme (grant agreement No. 945655).}

\medskip

We wish to thank S.\ Gallot, R.\ Neel, and E.\ Tr\'elat for stimulating discussions, and also Y.\ Colin de Verdière for his helpful suggestions that, in particular, led to Section~\ref{sec:concentration}. We also warmly thank L.\ de Haan for pointing out a technical result due to A.\ A.\ Balkema about de Haan functions contained in \cite[Appendix B]{deHaan}, key to the proof of Theorem \ref{t:weylexact-intro}. We finally thank V. Ivrii for his comments on a preliminary version of this work.

%%% Local Variables:
%%% mode: latex
%%% TeX-master: "Weyl-main"
%%% End:

%% file: heatestimates.tex
\section{Heat kernel estimates with remainder}\label{s:heat}

In this section, we prove on-diagonal estimates for the heat kernel and its trace on a compact Riemannian manifold with boundary, with explicit remainder control. We recall first some basic definitions valid in the more general non-complete setting.

\subsection{Notation and basic definitions}

For a Riemannian manifold $(M,g)$, possibly non-complete and with boundary, the injectivity radius from $p\in M$ is the supremum of lengths $\ell>0$ such that every geodesic of length smaller than $\ell$ emanating from $p$ is length-minimizing. The injectivity radius of $M$, denoted by $\inj(M)$, is then the infimum of the injectivity radius over $M$.
This definition extends the classical one. Observe that the exponential map $\exp_p : T_p M \to M$ is a diffeomorphism when restricted to any ball of radius smaller than the injectivity radius from $p$ and contained in the domain of the exponential map (geodesics cease to be defined when they hit the boundary or the metric boundary of the manifold).

We denote the Riemannian distance from $\partial M$ by $\distb:M\to [0,+\infty)$, that is
\begin{equation}
  \distb(p) = \inf_{q\in\partial M} d(p,q).
\end{equation}
A length-parametrized geodesic $\gamma:[0,t]\to M$, $\gamma(0)\in \partial M$ is length-minimizing from the boundary if for all $0\leq s<t$ it holds $\distb(\gamma(s))=s$. It follows that $\dot\gamma(0)\perp T_{\gamma(0)}\partial M$ and that $\gamma(0)$ is the only point of ${\partial M}$ realizing $\distb(\gamma(s))$  for $0\leq s<t$.
The injectivity radius from the boundary, denoted by $\inj_{\partial}(M)$, is then defined as $\inj(M)$ considering only length-minimizing geodesics from the boundary.

For a smooth function $f:M\to \R$, we let
\begin{equation}
  \Hess(f)(X,Y)=g(\nabla_X \nabla f, Y),\qquad X,Y\in \Gamma(M),
\end{equation}
where $\nabla$ denotes the covariant derivative. The notation $\Hess(f)\ge c$ (resp.\ $\le c$) for some constant $c\in\R$ is to be understood in the sense of quadratic forms and with respect to the metric $g$.
  The boundary $\partial M$ is convex (resp.\ strictly convex) if its second fundamental form $\Hess(\distb)|_{T \partial M}$  is non-positive (resp.\ negative). Moreover, it is mean convex if $\tr \Hess(\distb)|_{T \partial M} \le 0$.

\subsection{On-diagonal heat kernel estimates}

The Dirichlet and Neumann heat kernels $E^{-}$ and $E^+$ are the minimal fundamental solutions of the heat equation associated with the Laplace-Beltrami operator $\Delta$ with Dirichlet or Neumann boundary conditions  on $\partial M$. We denote the corresponding self-adjoint operators by $\Delta^-$ and $\Delta^+$, respectively.  The first result of the section is the following.  Henceforth, we use the convention that $1/0 = +\infty$.

\begin{theorem}\label{t:heat}
Let $(M,g)$ be a compact $n$-dimensional Riemannian manifold with convex boundary $\partial M$. 
Let $K\ge 0$ be such that $|\Sec(M)|\le K$. Moreover, let
    \begin{equation}\label{eq:def-rho}
      \rho(q) = \min\left\{\frac{\distb(q)}{2},  \inj(M) \right\}, \qquad \forall q\in M.
    \end{equation}
Then there exist constants $c_1,c_2,c_3>0$ depending only on $n$, such that
  \begin{equation}
\left|(4\pi t)^{n/2} E^{\pm}(t,q,q) -1\right| \leq c_1 K t+ c_2 e^{-c_3\tfrac{\rho(q)^2}{4t}},
\end{equation}
for any $q\in M$ and $t\in\R_+$ such that
\begin{equation}
  \sqrt{t} \leq \min\left\{\rho(q),\frac{\pi}{\sqrt{K}}\right\}.
\end{equation}
\end{theorem}
\begin{proof}
Consider the double $\bar M = M \cup_{\partial M} M$ of $M$, which is a compact smooth manifold without boundary, endowed with the Lipschitz metric $\bar g$ inherited from $g$. Let $\bar{d}$ and $\bar{\mu}$ denote the corresponding metric and measure on $\bar{M}$. Clearly, $\bar{d}$ and $\bar{\mu}$ coincide with $d$ and $\mu$, when restricted to either isometric copy $M \subset \bar{M}$. Following \cite{MS-curvature}, although the coefficients of the Laplace-Beltrami operator are discontinuous, there is a well-defined heat kernel $\bar E$ on $(\bar{M},\bar{g})$, which satisfies
  \begin{equation}\label{e:hk-double}
    E^\pm(t,p,q) = \bar E(t,p,q) \pm \bar E(t,p,q^*), \qquad \forall p,q \in M,
  \end{equation}
  where $q^*\in \bar M$ denotes the reflection of $q$ with respect to the boundary $\partial M \in \bar{M}$. We decompose the argument in several steps.

\paragraph{Step 1. Gromov-Hausdorff approximation.} For $\tau >0$, there exists a sequence $\bar g_\tau$ of smooth Riemannian metrics on $\bar{M}$ such that
   \begin{itemize}
    \item $(\bar{M},\bar{d}_\tau,\bar{\mu}_\tau) \to (\bar{M},\bar{d},\bar{\mu})$ in the measured Gromov-Hausdorff sense, as $\tau \to 0$;
    \item $\Ric(\bar g_\tau)\ge -K(n-1)$, for all $\tau>0$;
    \item for any compact set $K$ such that $K \cap \partial M = \emptyset$ and for sufficiently small $\tau$, we have $\bar g_\tau|_K = \bar g|_K$;
    \item the distance to $\partial M$ in $\bar M$ with respect to $\bar g_\tau$ coincides with $\distb$, as functions on $\bar{M}$.
      \end{itemize}
The construction of $(\bar g_\tau)_{\tau>0}$ is sketched in \cite[Sec.\ 4]{perelman} for positive Ricci curvature and strictly convex boundary, and can be extended to the case of convex boundary, see \cite[Thm.\ 1.8]{W-extension} and references therein. The measured Gromov-Hausdorff convergence in the sense of Fukaya \cite[Def.\ 0.2]{Fukaya} follows from the fact that $\bar{g}_\tau \to \bar{g}$ uniformly in coordinates.

As a consequence of the measured Gromov-Hausdorff convergence and the Ricci bound, we have that the corresponding heat kernels $\bar{E}^\tau$ satisfy
  \begin{equation}
    \lim_{\tau \downarrow 0}\bar E^\tau(t,p,q) = \bar E(t,p,q), \qquad \forall (t,p,q)\in \R_+\times \bar M\times \bar M,
  \end{equation}
uniformly on $\bar{M} \times \bar{M}$, for any fixed $t$. See \cite[Thm.\ 2.6]{Ding}. 

 We will now prove lower and upper bounds for $\bar E^\tau$ that are uniform with respect to $\tau$. Passing to the limit  and using \eqref{e:hk-double} will then yield the statement. 
  
\paragraph{Step 2. Lower bound.} The lower bound on $\bar E^\tau$ is a consequence of classical comparison theorems for the heat kernel on complete manifolds without boundary and Ricci lower bound, see e.g.\ \cite[Thm.~7, p.~196]{Chavel}. Namely, if we let $E_{-K}(t,r)$ be the heat kernel for the simply connected space form of constant curvature $-K$ we obtain $E_{-K}(t,\bar{d}_\tau(p,q)) \leq \bar{E}^\tau(t,p,q)$ for all $(t,p,q)\in \R_+\times \bar M\times \bar M$ and $\tau >0$. In particular, as $\tau \to 0$, we have
\begin{equation}\label{eq:lowerbound}
E_{-K}(t,\bar{d}(p,q)) \leq \bar{E}(t,p,q), \qquad \forall (t,p,q)\in \R_+\times \bar M\times \bar M.
\end{equation}

\paragraph{Step 3. Upper bound.} In this case, classical comparison theorems are only local. 
Nevertheless, we claim that there exists positive constants $c_1,c_2>0$, such that for any $o\in M$ and $\sqrt{t} <\min\{\rho(o),\pi/\sqrt{K}\}$, where $\rho(o)$ is defined in \eqref{eq:def-rho}, it holds
\begin{equation}\label{eq:upperbound}
  \bar E(t,o,o) \le E_K(t,0)+\frac{c_1}{t^{n/2}}e^{-c_2\frac{\rho^2(o)}{4t}}
  \qquad\text{and}\qquad
  \bar E(t,o,o^*) \le \frac{c_1}{t^{n/2}}e^{-c_2\frac{\rho^2(o)}{4t}}.
\end{equation}
Here, we denoted by $\bar{B}_o(r)$ (resp.\ $\bar{B}_o^\tau(r)$) the open ball with center $o$ and radius $r>0$ with respect to the metric $\bar{d}$ (resp.\ $\bar{d}_\tau$). When the ball is completely contained in one of the two copies of $M$ in $\bar{M}$, we drop the bar since no confusion arises.

Fix $o\in M$, and let $\rho=\rho(o)$. Let $\Omega = B_o(\rho/2)$. By definition of $\rho$, $\Omega=\bar{B}_o^\tau(\rho)$ is contained in one of the two copies of $M \subset \bar{M}$, and does not intersect $\partial M$, taking $\tau$ sufficiently small. Hence, we have $\bar{g}_\tau|_{\Omega} = g|_{\Omega}$, and
\begin{equation}\label{eq:Omegaset}
\Omega=\bar{B}_o^\tau(\rho/2) = \bar{B}_o(\rho/2) = B_o(\rho/2).
\end{equation}
Denote by $\bar{E}_{\Omega}^\tau(t,p,q)$ the heat kernel with respect to $\bar{g}_\tau$ on $\Omega$ with Dirichlet condition, which we set to zero if $p$ or $q\notin \Omega$. The Markov property of the heat kernel implies
\begin{equation}\label{eq:Grig}
  \bar E^\tau(t,o,q) \le \bar E^\tau_{\Omega}(t,o,q) + \sup_{\substack{0<s\le t\\p\in \partial\Omega}} \bar E^\tau(s,p,q).
\end{equation}
This follows, e.g., by applying \cite[Lemma~3.1]{Grigoryan2009} and upper-bounding the hitting probability appearing there by $1$.
We now estimate the two terms appearing on the r.h.s.\ of \eqref{eq:Grig}, which we will refer to as the local and the global term, respectively, for the cases $q =o$ and $q = o^*$.

Let us start by considering the local term. Since $\Omega \subset M$, it follows that $o^* \notin \Omega$, and hence, for $q=o^*$, we have
\begin{equation}\label{eq:local1}
  \bar{E}^{\tau}_{\Omega}(t,o,o^*)=0.
\end{equation}
Let now $q=o$. Since $\rho \leq \inj(M)$ and $\Omega = \bar{B}^\tau_o(\rho/2)$ lies in the region of $M$ where the metric is unperturbed, one has that $\Omega$ lies within the injectivity radius from $o$. Therefore, we can apply the comparison theorem \cite[Thm.\ 6, p.\ 194]{Chavel} and the domain monotonicity property of the Dirichlet heat kernel to obtain
\begin{equation}\label{eq:localtermtau}
\bar{E}^{\tau}_{\Omega}(t,o,o) \leq E_{K}(t,0).
\end{equation}

The global term in \eqref{eq:Grig} is more delicate. 
Observe that the Li-Yau inequality (see Lemma~\ref{l:li-yau}) requires only a lower bound on the Ricci curvature, and hence can be applied to the compact Riemannian manifold with no boundary $(\bar{M},\bar{g}_\tau)$, for which $\Ric(\bar{g}_\tau) \geq -(n-1)K$, for all $\tau>0$. As a consequence, there exist constants $C_1,C_2,C_3>0$, depending only on the dimension $n$ of $\bar{M}$, such that
\begin{equation}\label{eq:LY2}
\bar{E}^\tau(s,p,q) \leq \frac{C_1}{\sqrt{\vol^\tau(\bar{B}^\tau_p(\sqrt{s}))\vol^\tau(\bar{B}^\tau_q(\sqrt{s}))}} e^{C_2 K s - C_3 \frac{\bar{d}_\tau^2(p,q)}{4s}}, \quad \forall (s,p,q) \in \R_+ \times \bar{M} \times \bar{M}.
\end{equation}
%for all $(s,p,q) \in \R_+ \times \bar{M} \times \bar{M}$.

Recall that $p \in \partial \bar{B}_o^\tau(\rho/2)$, and $\rho=\rho(o) \leq \distb(o)$.
Therefore $\bar{d}_\tau(p,o) = \rho /2$, and
\begin{equation}
\bar{d}_\tau(p,o^*) \geq \bar{d}_\tau(o,o^*)-\bar{d}_\tau(p,o)= 2\distb(o)-\rho/2\geq 3\rho/2 \geq\rho.
\end{equation}
Hence, \eqref{eq:LY2}, for both cases $q\in \{o,o^*\}$, yields
\begin{equation}\label{eq:LY3}
\bar{E}^\tau(s,p,q) \leq \frac{C_1}{\sqrt{\vol^\tau(\bar{B}^\tau_p(\sqrt{s}))\vol^\tau(\bar{B}^\tau_q(\sqrt{s}))}} e^{C_2 K s - C_3 \frac{\rho^2(o)}{4s}}, \qquad q \in \{o,o^*\}.
\end{equation}

Recall now that in \eqref{eq:Grig} $s \leq t$. Furthermore $\sqrt{t}\leq \rho = \min\big\{\distb(o), \inj(M)\big\}$. It follows that $\bar{B}_q^\tau(\sqrt{s/2})$, for $q \in \{o,o^*\}$, does not intersect $\partial M$, and hence, we can choose $\tau$ sufficiently small so that these sets lie in the region of $\bar{M}$ where the metric is unperturbed, yielding
\begin{equation}\label{eq:vol1}
\vol^\tau(\bar{B}_q(\sqrt{s /2})) = \vol(B_q(\sqrt{s/2})),  \qquad \forall s \leq t, \quad q\in \{o,o^*\}.
\end{equation}
Furthermore, since $\sqrt{t} \leq \inj(M)$, and thanks to the upper bound on the sectional curvature of $(M,g)$, we can bound from below the r.h.s.\ of \eqref{eq:vol1} with the volume of the analogue ball in the simply connected space form of curvature $K$, yielding
\begin{equation}\label{eq:vol2}
\vol^\tau(\bar{B}_q(\sqrt{s /2})) \geq \vol(B_K(\sqrt{s /2})),  \qquad \forall s \leq t, \quad q\in \{o,o^*\}.
\end{equation}
Finally, since $\sqrt{t} \leq \tfrac{\pi}{\sqrt{K}}$, we deduce (see Lemma~\ref{l:deduction}) the existence of a constant $C>0$ depending only on $n$ such that, for $\tau$ sufficiently small, it holds
\begin{equation}\label{eq:vol}
  \vol^\tau(\bar{B}_q^\tau(\sqrt{s}))\ge { \vol(B_K(\sqrt{s /2})) \geq } C s^{n/2},\qquad \forall s\le t, \quad q\in\{o,o^*\}.
\end{equation}
The same argument shows that \eqref{eq:vol} holds also when replacing $q$ with $p \in\partial\Omega$ for $\tau$ small. By plugging \eqref{eq:vol} in \eqref{eq:LY3}, using again that $Ks\le \pi^2$, we deduce that
\begin{equation}\label{eq:LY4}
\bar{E}^\tau(s,p,q) \leq \frac{c_1}{s^{n/2}} e^{- c_2 \frac{\rho^2(o)}{4s}},\qquad \forall s\le t, \quad q\in\{o,o^*\}.
\end{equation}
Up to enlarging the constant $c_1$ (still depending only on $n$), one has
\begin{equation}\label{eq:global}
\sup_{\substack{0<s\le t\\p\in \partial\Omega}} \bar E^\tau(s,p,q) \leq  \frac{c_1}{t^{n/2}} e^{-c_2\tfrac{\rho^2(o)}{4t}}, \qquad  q \in \{o,o^*\},
\end{equation}
which is the the final estimate for the global part of \eqref{eq:Grig}.

By \eqref{eq:local1} (resp.\ \eqref{eq:localtermtau}) and \eqref{eq:global}, passing to the limit as $\tau\to 0$ in \eqref{eq:Grig}, completes the proof of the upper bounds \eqref{eq:upperbound}.

\paragraph{Step 4. Conclusion.}
By \eqref{e:hk-double}, the lower bound \eqref{eq:lowerbound} and the upper bound \eqref{eq:upperbound} for the heat kernel on the double $\bar{M}$ yield the following on-diagonal estimates for the Dirichlet and Neumann heat kernels of the original manifold with boundary:
\begin{equation}
E_{-K}(t,0) - \frac{C_1}{t^{n/2}} e^{-C_3\tfrac{\rho^2(o)}{4t}} \leq E^{\pm}(t,o,o) \leq E_{K}(t,0)+ \frac{2C_1}{t^{n/2}} e^{-C_3\tfrac{\rho^2(o)}{4t}},
\end{equation}
valid for all $0<\sqrt{t} \leq \min\{\rho(o),\tfrac{\pi}{\sqrt{K}}\}$. We conclude by using the uniform estimates of the model kernels $E_{\pm K}(t,0)$ given in  Lemma~\ref{l:pKreste} (which we apply with $T= \pi^2$).
\end{proof}

\subsection{Heat trace bound}

In this section we apply Theorem~\ref{t:heat} to estimate the heat trace on $M$.
\begin{theorem}\label{t:traceheat}
Let $(M,g)$ be a smooth compact $n$-dimensional Riemannian manifold with convex boundary $\partial M$. 
Let $K \geq 0$ such that $|\Sec(M)|\le K$.
Then there exists a constant $c>0$, depending only on $n$, such that the following estimate for the Dirichlet or Neumann heat kernels holds:
\begin{equation}\label{eq:integralestimate}
\left|\frac{(4\pi t)^{n/2}}{\vol(M)\phantom{}}\int_M E^{\pm}(t,q,q) d\mu_g(q) - 1\right|  \leq  c \left(\frac{t}{t_0}\right)^{1/2},
\end{equation}
for all values of $t \in \R_+$ such that
\begin{equation}\label{eq:t0}
\sqrt{t}\leq \sqrt{t_0}= \min\left\{\inj(M),\frac{\inj_\partial(M)}{4}, \frac{\pi}{\sqrt{K}}\right\}.
\end{equation}
\end{theorem}
\begin{proof}
Fix $t$ as in our assumptions. Let $\mathfrak{i}= \min\{\inj(M),\tfrac{\inj_\partial(M)}{4}\}$. We split $M$ into $3$ disjoint components (see Figure~\ref{f:split}):
\begin{align}
\Omega_1  = \left\{\distb < \sqrt{t} \right\}, \qquad \Omega_2  = \left\{ \sqrt{t} \leq \distb < \mathfrak{i} \right\}, \qquad \Omega_3  = \left\{ \mathfrak{i} \leq \distb \right\}.
\end{align}
We estimate the heat trace on these three sets separately.

\begin{figure}
\centering
\includegraphics[width=12cm]{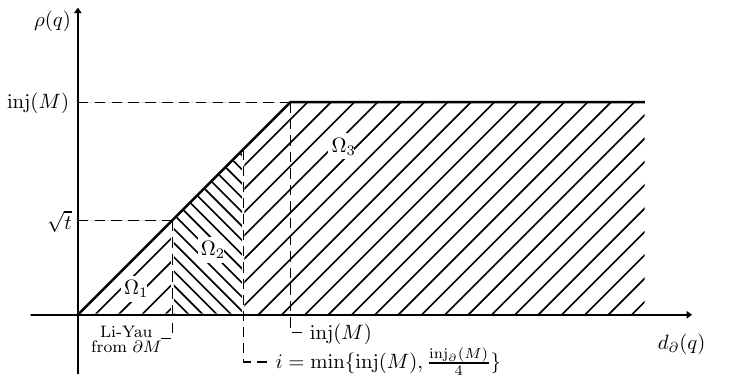}
\caption{The regions $\Omega_1,\Omega_2,\Omega_3$. The condition $\sqrt{t} \leq \inj(M)$ ensures the existence of $\Omega_2$, $\Omega_3$ where we can apply Theorem \ref{t:heat}. The condition $\sqrt{t}\leq \min \left\{\inj(M),\tfrac{\pi}{\sqrt{K}},\tfrac{\inj_\partial(M)}{4}\right\}$ allows one to apply the Li-Yau estimate on $\Omega_1$.}\label{f:split}
\end{figure}

\textbf{Estimate on $\Omega_1$}. By definition, and thanks to our assumption on $t$, we have
\begin{equation}
\distb < \sqrt{t} \leq \min\left\{\inj(M),\frac{\inj_\partial(M)}{4}, \frac{\pi}{\sqrt{K}}\right\}.
\end{equation}
It follows that $\rho(q) = \distb(q)$ for $q\in\Omega_1$, where $\rho$ is defined in \eqref{eq:def-rho}, and $\distb(q) < \inj_\partial(M)/4$.

Observe that, by construction, $\sqrt{t} > \rho(q)$, and one cannot apply the bound of Theorem~\ref{t:heat}. However, the assumption on $t$ allows one to apply the Li-Yau type estimate \eqref{eq:LY-diag} of Lemma~\ref{l:li-yau}. This yields,
\begin{equation}\label{eq:kh-omega1}
 \int_{\Omega_1} |(4\pi t)^{n/2}E^{\pm}(t,q,q) -1|d\mu_g(q) \leq C_4 \vol(\Omega_1).
\end{equation}
In addition, we have
\begin{equation}\label{eq:vol-omega1}
\vol(\Omega_1) =  \int_0^{\sqrt{t}} \vol(Z_x)   d x,
\end{equation}
where $\vol(Z_x)$ denotes the Riemannian volume of the level set $Z_x=\{\distb = x\}$ (a smooth $(n-1)$-dimensional hypersurface {for $x>0$}, with $Z_0 = \partial M$). By Lemma \ref{l:stdcomparisonforH}, on the region $\distb<\inj_\partial(M)/2$ it holds $|\Hess(\distb)|\leq H$ for some $H$ satisfying $1\leq H\sqrt{t}_0 \leq 10$. It holds
\begin{equation}
\frac{d}{d x} \vol(Z_x) = \int_{Z_x} \tr \Hess(\distb)\, d \sigma_x \leq (n-1)H \vol(Z_x), \qquad \forall x< \inj_\partial(M)/2,
\end{equation}
which implies 
\begin{equation}\label{eq:der-z}
\vol(Z_x) \leq \vol(\partial M) e^{(n-1) H x}, \qquad \forall x< \inj_\partial(M)/2.
\end{equation}
Using the fact that $H\sqrt{t} \leq H \sqrt{t_0}\leq 10$, and plugging \eqref{eq:der-z} into \eqref{eq:vol-omega1} we conclude the estimate on $\Omega_1$, which yields together with \eqref{eq:kh-omega1},
\begin{equation}\label{eq:Omega1}
 \int_{\Omega_1} |(4\pi t)^{n/2}E^{\pm}(t,q,q) -1|d\mu_g(q) \leq  c\, \vol(\partial M) \sqrt{t}.
\end{equation}
for some constant $c>0$ depending only on $n$.

\textbf{Estimate on $\Omega_2$}. By construction, $\Omega_2$ still lies in the region within the injectivity radius from $\partial M$. Furthermore, it still holds $\rho(q) = \distb(q)$ for $q\in\Omega_2$. Here, however, $\sqrt{t} \leq  \min\{\rho(q),\tfrac{\pi}{\sqrt{K}}\}$, and hence we can apply the result of Theorem~\ref{t:heat}. Henceforth, $c$ denotes a positive constant depending only the dimension, whose value can change at each step of the computation. We have
\begin{align}
\int_{\Omega_2}\left| (4\pi t)^{n/2} E^{\pm}(t,q,q)-1 \right| d\mu_g(q) & \leq \int_{\Omega_2}\left( c K t  +c e^{-\tfrac{\rho^2(q)}{c t}} \right)d\mu_g(q) \\
& = c K \vol(\Omega_2) t + c \int_{\Omega_2} e^{-\tfrac{\distb(q)^2}{c t}} d\mu_g(q) \\
& = c K \vol(\Omega_2) t + c \int_{2\sqrt{t}}^{\mathfrak{i}}e^{- \tfrac{x^2}{c t}} \vol(Z_x) dx \\
& \leq c K \vol(\Omega_2) t + c \vol(\partial M) \int_{0}^{\infty }e^{-\tfrac{x^2}{c t}+(n-1)H x}  dx \\
& \leq c K \vol(\Omega_2) t + c \vol(\partial M) \sqrt{t}. \label{eq:Omega2}
\end{align}
Here, we used the fact that $\Omega_2$ lies within the region where the estimate \eqref{eq:der-z} holds. Furthermore, we evaluate explicitly the Gaussian integral in the last inequality, and use the fact that $H\sqrt{t}  \leq 10$.

\textbf{Estimate on $\Omega_3$}. For $q\in\Omega_3$, it does not necessarily hold $\rho(q) = \distb(q)$, neither $q$ is necessarily within the injectivity radius from $\partial M$. However, it holds $\rho(q)\ge \mathfrak{i} \geq \sqrt{t}$, and we can apply Theorem~\ref{t:heat} again. Hence, we obtain
\begin{align}
\int_{\Omega_3}\left| (4\pi t)^{n/2} E^{\pm}(t,q,q)-1 \right| d\mu_g(q) & \leq \int_{\Omega_3}\left( c K t  +c e^{-\tfrac{\rho^2(q)}{c t}} \right)d\mu_g(q) \\
& \leq \left(c K  t + c e^{-\tfrac{\mathfrak{i}^2}{c t}}\right) \vol(\Omega_3) \\
& \leq c \left( K  t + \tfrac{t}{\mathfrak{i}^{2}}\right)\vol(\Omega_3). \label{eq:Omega3}
\end{align}
Here, in the last step, we used the inequality $ e^{-1/x} \leq x/e$ for $x>0$.

Since $\vol(\Omega_i)/\vol(M) \leq 1$, splitting the l.h.s.\ of \eqref{eq:integralestimate} in the subsets $\Omega_1,\Omega_2,\Omega_3$, using \eqref{eq:Omega1}, \eqref{eq:Omega2}, \eqref{eq:Omega3}, and increasing the constants, yields
\begin{align}
\left|\frac{(4\pi t)^{n/2}}{\vol(M)\phantom{}}\int_M E^{\pm}(t,q,q) d\mu_g(q) - 1\right| & \leq c\left( \frac{\vol(\partial M)}{\vol(M)} \sqrt{t} +K t + \frac{t}{\mathfrak{i}^2}\right) \\
& \leq c\left( \frac{\vol(\partial M)}{\vol(M)} \sqrt{t} + \frac{t}{t_0}\right), \label{eq:stimaintermedia}
\end{align}
where we used the definition of $t_0$. It remains to estimate the ratio $\vol(\partial M)/\vol(M)$ in \eqref{eq:stimaintermedia}. Proceeding as in \eqref{eq:der-z}, but using this time the lower bound on the Hessian, we obtain the corresponding lower bound
\begin{equation}
\vol(Z_x) \geq \vol(\partial M) e^{-H(n-1)x}, \qquad \forall x  < \inj_\partial(M)/2.
\end{equation}
Therefore, since $t_0 \leq \mathfrak{i}^2$ and $H \sqrt{t_0} \leq 10$, we have
\begin{equation}
\frac{\vol(M)}{\vol(\partial M)} \geq \int_0^{\sqrt{t_0}/(n-1)} \frac{\vol(Z_x)}{\vol(\partial M)} dx  \geq \frac{ 1- e^{-H \sqrt{t_0}}}{H(n-1)} \geq \frac{ \sqrt{t_0}(1-e^{-10}) }{10(n-1)} .
\end{equation}
By plugging this estimate in \eqref{eq:stimaintermedia} and since $t \leq t_0$, one gets the result.
\end{proof}

%The remainder estimate of Theorem \ref{t:traceheat} can be expressed in many equivalent ways. For example, one can control the ratio $\vol(\partial M)/\vol(M)$ in terms of the lower bound on the mean curvature of the boundary and the lower bound on the injectivity radius from the boundary (and hence $\mathfrak{i}$). We prefer \eqref{eq:integralestimate} since it is closer to the well known small time heat kernel asymptotics for Riemannian manifolds with boundary. 

The next corollary is a version of Theorem~\ref{t:traceheat}, global with respect to $t \in \R_+$.
\begin{cor}\label{c:traceheat}
In the setting of Theorem \ref{t:traceheat}, there exists a constant $c>0$, depending only on the dimension $n$, such that
\begin{equation}\label{eq:integralestimate2}
\left|\frac{(4\pi t)^{n/2}}{\vol(M)\phantom{}}\int_M E^{\pm}(t,q,q) d\mu_g(q) - 1\right|  \leq  c \left[ (t/t_0)^{1/2} + (t/t_0)^{n/2}\right] ,
\end{equation}
where $\sqrt{t_0}= \min\left\{\inj(M),\frac{\inj_\partial(M)}{4}, \frac{\pi}{\sqrt{K}}\right\}$.
\end{cor}

\begin{proof}
In the r.h.s.\ of \eqref{eq:integralestimate2}, the first and second term control the remainder for small and large times, respectively. If $t\leq t_0$ the result follows from Theorem \ref{t:traceheat}. If $t \geq t_0$, set $W(t) = \tfrac{(4\pi)^{n/2}}{\vol(M)}\int_M E^\pm(t,q,q) d\mu_g(q)$. Since $W(t)$ is decreasing and positive, we have
\begin{align}
\left\lvert W(t)-t^{-n/2}\right\rvert & \leq  W(t_0) + t_0^{-n/2} \\
& \leq \left| W(t_0)-t_0^{-n/2} \right| + 2 t_0^{-n/2} \\
& \leq t_0^{-n/2}|t_0^{n/2}W(t_0)-1| + 2 t_0^{-n/2} \\
& \leq t_0^{-n/2}(c + 2),
\end{align}
where we used Theorem \ref{t:traceheat} at $t=t_0$.
\end{proof}

%\begin{rmk}\label{rmk:removeH}
%\blue Given a lower bound $\Sec(M)\geq -K$ and convexity of $\partial M$, one can deduce $|\Hess(\distb)|\leq H_0$ on the region $\{\distb< \tfrac{\inj_{\partial}(M)}{2}\}$, for some $H_0 = H_0(K,\inj_\partial(M),\inj(M))$, such that (cf.\ Lemma \ref{l:stdcomparisonforH}).
%\begin{equation}
% \frac{1}{1000} \min \left\{\inj(M),\frac{\inj_\partial(M)}{4},\frac{\pi}{\sqrt{K}}\right\} \leq \frac{1}{H_0} \leq \min \left\{\inj(M),\frac{\inj_\partial(M)}{4},\frac{\pi}{\sqrt{K}}\right\}.
%\end{equation}
%As a consequence, the upper bound $H$ in Theorem \ref{t:traceheat}, Corollary \ref{c:traceheat} and the forthcoming Theorem \ref{t:weylremainder} can be omitted, up to re-defining $\sqrt{t_0}:= \tfrac{1}{1000}\min\left\{\inj(M),\frac{\inj_{\partial}(M)}{4}, \frac{\pi}{\sqrt{K}}\right\}$.
%\end{rmk}

\subsection{Weyl's law with remainder}

When $M$ is compact,  the spectrum of $-\Delta^\pm_\Omega$ is a discrete subset of the positive real axis, i.e., $\sigma(-\Delta^\pm)\subset [0,+\infty)$, accumulating at infinity.  The eigenvalue counting function is
\begin{equation}
  N^{\pm}(\lambda):=\#\{\sigma(-\Delta^\pm)\cap [0,\lambda]\}.
\end{equation}

It is well-known that heat trace asymptotics imply asymptotics for $N(\lambda)$, by means of Tauberian theorems in the form of Karamata \cite{karamata}. We need here a Karamata type result with remainder, due to Freud. See \cite[Thm.\ B]{Tauberian2} or \cite[Thm.\ 3.1]{Korevaarbook} and references within. Since for our purposes we need to know the explicit dependence of the constants with respect to all parameters and functions at play, the statement below is slightly more precise than the one in \cite{Tauberian2}. However, the proof is unchanged and we omit it.

\begin{theorem}[Freud's Tauberian Theorem \cite{Tauberian2}]\label{t:ingham}
Let $\mu : [0,\infty) \to \R$ be a positive and non-decreasing function. Denote by the same symbol the associated Stieltjes measure. Let $\alpha>-1$, and let $\chi : [0,+\infty)\to \R$ be a function such that
\begin{itemize}
\item $\chi(\lambda)>0$ for all $\lambda >0$;
\item $\lambda \mapsto \chi(\lambda)$ is increasing and tends to $\infty$ as $\lambda \to +\infty$;
\item $\lambda \mapsto \lambda^{-\alpha-1}\chi(\lambda)$ is decreasing.
\end{itemize}
Let $\hat{\mu}(t) = \int_0^\infty e^{-t\lambda} d\mu(\lambda)$ denote the Laplace transform of $\mu$. Suppose that there exists $c>0$ such that
\begin{equation}
\left\lvert t^\alpha \hat\mu(t)-  1 \right\rvert \leq \frac{c}{\chi(1/t)}, \qquad \forall t>0.
\end{equation}
Then there exists another constant $C=C(c,\alpha)>0$ such that
\begin{equation}
\left\lvert \frac{\Gamma(\alpha+1) \mu(\lambda)}{\lambda^\alpha} - 1\right\rvert \leq \frac{C}{\log(\chi(\lambda)+1)} , \qquad \forall \lambda>0.
\end{equation}
\end{theorem}
\begin{rmk}
Theorem \ref{t:ingham} in particular recovers the classical statement of Karamata (cf. \cite{regularvariation}): if $\hat{\mu}(t) \sim t^{-\alpha}$ as $t \to 0$, then $\mu(\lambda) \sim \lambda^{\alpha}/\Gamma(\alpha +1)$ as $\lambda \to \infty$.
\end{rmk}

We now use Corollary \ref{c:traceheat} to derive the Weyl's law with remainder for $M$.
\begin{theorem}\label{t:weylremainder}
 Let $(M,g)$ be a compact $n$-dimensional Riemannian manifold with convex boundary $\partial M$. Let $K\geq 0$ such that $|\Sec(M)|\le K$. Then, there exists a constant $C>0$, depending only on $n$, such that the following estimate holds for the counting function for Dirichlet or Neumann eigenvalues:
\begin{equation}
\left\lvert\frac{N(\lambda)}{\tfrac{\omega_n}{(2\pi)^n}\vol(M)\lambda^{n/2}}-1\right\rvert \leq \frac{C}{\log(1+\sqrt{\lambda/\lambda_0})}, \qquad \forall \lambda >0.
\end{equation}
with
\begin{equation}\label{eq:lambda0}
\sqrt{\lambda_0} = \frac{1}{ \min\left\{\inj(M),\frac{\inj_\partial(M)}{4}, \frac{\pi}{\sqrt{K}}\right\}}.
\end{equation}
\end{theorem}
\begin{proof}
The proof is an application of Theorem \ref{t:ingham}. Let
\begin{equation}
\mu(\lambda) = \frac{(4\pi)^{n/2}}{\vol(M)} N(\lambda).
\end{equation}
The Laplace transform of the corresponding measure satisfies
\begin{equation}
\hat{\mu}(t)=  \frac{(4\pi)^{n/2}}{\vol(M)} \int_0^\infty e^{-t\lambda} d N(\lambda) = \frac{(4\pi)^{n/2}}{\vol(M)} \sum_{i=1}^\infty e^{-t \lambda_i} = \frac{(4\pi)^{n/2}}{\vol(M)} \int_M E^{\pm}(t,q,q) d\mu_g.
\end{equation}
By Corollary~\ref{c:traceheat}, $\hat{\mu}$ satisfies the assumptions in Theorem~\ref{t:ingham} with $\alpha = n/2$, and $\chi(\lambda) = (\lambda/\lambda_0)^{1/2} + (\lambda/\lambda_0)^{n/2}$. We conclude by recalling that $\Gamma(\tfrac{n}{2}+1) = \tfrac{\pi^{n/2}}{\omega_n}$.
\end{proof}

%The upper bound in Theorem \ref{t:weylremainder} can be seen as a version of Buser's inequality whose leading order constant is sharp as $\lambda \to \infty$. The price to pay for sharpness are stronger curvature assumptions, since the classical Buser's inequality for closed manifolds only requires a Ricci lower bound \cite{Buser}. In the following, we formulate the upper bound of Theorem \ref{t:weylremainder} in a weaker form, more similar to  Buser's inequality.
%
%\begin{corollary}\label{c:Buser}
%In the setting of Theorem \ref{t:traceheat}, there exists a constant $C>0$ depending only on $n$, such that the following estimate holds for the counting function for the Dirichlet or Neumann eigenvalues:
%\begin{equation}\label{eq:Buser}
%N(\lambda) \leq C \vol(M) \left( \lambda^{n/2} + \inj(M)^{-n} + \inj_\partial(M)^{-n} + K^{n/2}+H^{n}\right), \qquad \forall \lambda > 0.
%\end{equation}
%\end{corollary}
%{\blue Classical comparison estimates show that, under the curvature bounds of Theorem \ref{t:traceheat}, $H$ can be estimated in terms of $K,\inj_\partial(M)$ and $\inj(M)$, so that the term containing $H$ in \eqref{eq:Buser} can be actually dispensed of.}
%
%{\red See also \cite[Sec.\ 1.2]{Hassannezhad} for more details on Buser's and related inequalities. }
%

%%% Local Variables:
%%% mode: latex
%%% TeX-master: "Weyl-main"
%%% End:

%% file: weyllaw.tex
\section{Geometric structure at the singularity}\label{s:singular}

In this section we collect some preliminary results on non-complete Riemannian manifolds $\M$ satisfying Assumption \ref{a:singularity}. The distance from the metric boundary is denoted by $\distmb$. It is a Lipschitz function, and it satisfies the eikonal equation $|\nabla\distmb|\equiv 1$. For any $0\le a<b\le +\infty$, we let $\M_a^b=\{a\leq \distmb \leq b\}\cap \M$. The following lemma collects some basic properties of $\M$.
 
\begin{lemma}\label{l:basic}
  Let $\M$ be a Riemannian manifold with compact metric completion and satisfying Assumption \ref{a:singularity}. Then there exists $C>0$ such that, for all $0<\varepsilon<\varepsilon_0/2$, the compact manifold with convex boundary $\M_\varepsilon^\infty$ satisfies the following bounds
   \begin{equation}
      \inj_{\partial}(\M_\varepsilon^\infty)\ge \frac{\varepsilon_0}{2}, \qquad \inj(\M_\varepsilon^\infty) \geq \frac{\varepsilon}{C}, \qquad
    |\Sec(\M_\varepsilon^\infty)| \leq \frac{C}{\varepsilon^2},
    \end{equation}
\end{lemma}
\begin{proof}
On $U = \{\delta <\varepsilon_0\}$, the distance from the metric boundary $\distmb$ is smooth. Let $\distb^\epsilon$ the distance on $\M_\varepsilon^\infty$ from $\partial\M_\varepsilon^\infty$. It holds $\distb^\epsilon = \distmb -\epsilon$. For $\epsilon< \epsilon_0/2$, we immediately obtain the bound on $\inj_{\partial}(\M_\varepsilon^\infty)$, and the convexity of $\partial\M_\varepsilon^\infty$. The bounds on $\inj(\M_\varepsilon^\infty)$ and on $\Sec(\M_\varepsilon^\infty)$ follow from the corresponding ones in Assumption~\ref{a:singularity}.
\end{proof}

\begin{prop}\label{prop:inj-strict-convex}
Let $\M$ be a Riemannian manifold with compact metric completion and satisfying Assumption \ref{a:singularity}. If the convexity condition (b) is assumed to be strict, or the metric \eqref{eq:metricnormalform} is of warped product type on a neighborhood of the singularity, then the injectivity radius condition (d) is automatically verified.
\end{prop}

\begin{proof}
Let $p,q \in \M$. Let $\gamma : [0,1] \to \M$ be a piecewise smooth curve joining $p$ and $q$. Assume that for some non-empty interval $I \subset [0,1]$ it holds
\begin{equation}\label{eq:property}
\delta(\gamma(t)) < \min \{\delta(p),\delta(q),\varepsilon_0\}, \qquad \forall t \in I.
\end{equation}
Indeed, $\gamma|_I$ has positive length. We can assume $I$ to be maximal with respect to \eqref{eq:property}, in which case $\delta(\gamma(\cdot))$ is constant on $\partial I$. By construction, $\gamma(I) \subseteq U \simeq (0,\varepsilon_0) \times Z$, with $g = dx^2 + h(x)$, where $h(x)$ is a one-parameter family of smooth metrics on $Z$. Since $\Hess(\delta) \leq 0$, and $\delta(x,z) = x$, it follows that $x\mapsto h(x)$ is non-increasing. Thus, replacing on $I$ the curve $\gamma(t) = (x(t),z(t))$ with its projection $(x(\partial I),z(t))$ we obtain a shorter piecewise smooth curve between $p$ and $q$. It follows that,  when looking for length-minimizers between $p$ and $q$, we can restrict to curves such that
\begin{equation}\label{eq:vicinanza}
\delta(\gamma(t)) \geq \min \{\delta(p),\delta(q),\varepsilon_0\}, \qquad \forall t \in [0,1],
\end{equation}
which are separated from the metric boundary of $\M$. It follows that for any $p,q \in \M$ there exist a minimizing curve joining them, any such a curve is a Riemannian geodesic, and  any such geodesic respects \eqref{eq:vicinanza}.

In particular, if $p,q\in \M_\varepsilon^\infty$ and $0<\varepsilon < \varepsilon_0$, there exists a minimizing geodesic of $\M$ joining them, which is entirely contained in $\M_\varepsilon^\infty$.  Taking into account the definition of injectivity radius of a manifold with boundary, the proof of the classical Klingenberg Lemma (cf.\ \cite[Ch.\ 5]{CheegerEbin}) holds unchanged, yielding
\begin{equation}\label{eq:Klingenberg}
\inj(\M_\varepsilon^\infty) \geq \min \left\lbrace \frac{\pi}{\sqrt{K_\varepsilon}}, \frac{\ell_\varepsilon}{2}\right\rbrace.
\end{equation}
Here $K_\varepsilon = C/\varepsilon^2$ is the upper bound on the sectional curvature of $\M_\varepsilon^\infty$, and $\ell_\varepsilon$ is the length of the shortest simple closed geodesic in $\M_\varepsilon^\infty$.

Let $\gamma : [0,1] \to \M$ be such a shortest closed geodesic. Let $\gamma(t_0)$ be a point of closest distance from the metric boundary, and assume now that the convexity assumption in (b) is strict. If $t_0 <\varepsilon_0$, we have
\begin{equation}\label{eq:convexitystrict}
(\delta \circ \gamma)^{\prime\prime}(t_0) = \Hess(\distmb)(\dot\gamma(t_0),\dot\gamma(t_0)) < 0.
\end{equation}
This is a contradiction. It follows that $\delta(\gamma(t)) \geq \varepsilon_0$ for all $t\in [0,1]$, and the length of the closed geodesic in \eqref{eq:Klingenberg} does not depend on $\varepsilon$. We conclude by \eqref{eq:Klingenberg}.

If the convexity in (b) is not strict, we avoid the contradiction only if $\gamma$ lies in a level set of $\delta$, that is $\gamma(t) = (\eta,\hat{\gamma}(t))$, for some $\eta \in (\varepsilon,\varepsilon_0)$. Assume in this case that the metric is of warped product type on the neighborhood $U \simeq (0,\varepsilon_0)\times Z$, that is
\begin{equation}
g = dx^2 + f^2(x) \hat{h}, \qquad f : (0,\varepsilon_0) \to \R,
\end{equation}
where $\hat{h}$ is a fixed Riemannian metric on $Z$. It follows that $\hat{\gamma}: [0,1] \to Z$ is a closed geodesic in $(Z,f^2(\eta)\hat{h})$. The convexity assumption implies that $f$ is non-increasing, therefore $\ell(\hat{\gamma})$ cannot be smaller than the shortest simple closed geodesic of $(Z,f^2(\varepsilon_0)\hat{h})$, which does not depend on $\varepsilon$. We conclude again by \eqref{eq:Klingenberg}.
\end{proof}

 We will need the following simple estimate.

\begin{lemma}\label{l:vol}
 Let $\M$ be a Riemannian manifold with compact metric completion and satisfying Assumption \ref{a:singularity}. Then there exists $C>0$ such that
\begin{equation}\label{eq:volineq}
\frac{\vol(\M_{b}^\infty)}{\vol(\M_{a}^\infty)}  \geq  \left(\frac ab\right)^{1/C}, \qquad \forall\, 0<a \leq b \leq \frac{\varepsilon_0}{2}.
\end{equation}
\end{lemma}
\begin{proof}
Since close to the metric boundary the metric has the form \eqref{eq:metricnormalform}, we have
\begin{equation}\label{eq:volproof1}
\vol(\M_\varepsilon^\infty) =  \int_{\varepsilon}^{\varepsilon_0} \vol(\partial \M_x^\infty)\, dx+ \vol(\M_{\varepsilon_0}^\infty).
\end{equation}
Lemma \ref{l:stdcomparisonforH} implies that $\Delta \distmb   \ge -\tfrac{C(n-1)}{x}$ for $x\leq\epsilon_0/2$. Hence it holds
\begin{equation}
\frac{d}{dx} \vol(\partial \M_x^\infty) = \int_{\partial \M_x^\infty} \Delta \distmb \, d\sigma_x  \geq -\frac{C(n-1)}{x}\vol(\partial \M_x^\infty), \qquad \forall x \leq \frac{\varepsilon_0}{2}.
\end{equation}
By Gronwall's Lemma, this yields
\begin{equation}\label{eq:volproof2}
\frac{\vol(\partial \M_x^\infty)}{\vol(\partial \M_\varepsilon^\infty)} \geq \left(\frac{\varepsilon}{x}\right)^{C(n-1)}, \qquad \forall  x \in [\varepsilon,\varepsilon_0/2].
\end{equation}
Combining \eqref{eq:volproof1} with \eqref{eq:volproof2} we obtain
\begin{equation}
\frac{\vol(\M_\varepsilon^\infty)}{\vol(\partial \M_\varepsilon^\infty)}  \geq \int_\varepsilon^{\varepsilon_0}\left(\frac{\varepsilon}{x}\right)^{C(n-1)} dx  =\varepsilon \int_1^{\varepsilon_0/\varepsilon} \left(\frac{1}{x}\right)^{C(n-1)} dx.
\end{equation}
In particular, there exists $C'>0$ such that
\begin{equation}
\frac{d}{d\varepsilon} \log \vol(\M_\varepsilon^\infty) \geq -\frac{1}{C' \varepsilon}, \qquad \forall \varepsilon\leq \varepsilon_0/2,
\end{equation}
which yields \eqref{eq:volineq} upon integration.
\end{proof}

\section{Weyl's asymptotics for singular manifolds}\label{s:weyl}

In this section we prove Theorem~\ref{t:weyl-intro}, which we recall for the readers convenience. 

\begin{theorem}\label{t:weyl}
  Let $\M$ be a $n$-dimensional Riemannian manifold with compact metric completion and satisfying Assumption~\ref{a:singularity}. Then there exist $C_{\pm}>0$ and $\Lambda>0$ such that
\begin{equation}\label{eq:weylbounds1}
    C_-\le \frac{N(\lambda)}{\lambda^{n/2}\vol\left(\M_{1/\sqrt{\lambda}}^\infty\right)}\le C_+,\qquad \forall \lambda \geq \Lambda.
\end{equation}
\end{theorem}

We now introduce a precise definition of Dirichlet/Neumann extensions in the singular setting. For a domain $\Omega\subset\M$, the Friedrichs (or Dirichlet) Laplace-Beltrami operator $\Delta^-_{\Omega}$ is the self-adjoint operator on $L^2(\Omega,d\mu_g)$ associated with the quadratic form
\begin{equation}
  Q(u) = \int_{\Omega} |\nabla u|^2\,d\mu_g,
\end{equation}
with domain $H^1_0(\Omega)$, i.e., the closure of $C^\infty_c(\Omega)$ w.r.t.\ the norm $\|\cdot\|_{1}=\|\cdot\|_{L^2(\Omega,d\mu_g)} + Q(\cdot)^{1/2}$. On the other hand, we let the Neumann Laplace-Beltrami operator $\Delta^+_{\Omega}$ be the operator associated with $Q$ with domain $H^1(\Omega)$, i.e., the closure w.r.t. $\|\cdot\|_{1}$ of the space $C^\infty(\Omega)$ of functions made of the restrictions to $\Omega$ of functions in $C^\infty_c(\M)$.
\begin{rmk}
By this definition, constant functions are not in the domain of $\Delta^+_\Omega$, and in particular $0$ is not an eigenvalue, when $\Omega$ is adjacent to the singularity, i.e., the metric boundary of $\M$. Our definition of $\Delta^+_\Omega$, roughly speaking, imposes Neumann conditions only where $\partial\Omega$ is not adjacent to the singularity.
\end{rmk}

Particularly relevant will be the cases $\Omega = \M_a^b$ with $0\le a<b\le\infty$. The next results are proved in Appendix \ref{app:compactness} in a more general setting.

\begin{prop}[Compactness of the resolvent]
 Let $\M$ be a non-complete Riemannian manifold with compact metric completion and satisfying Assumption \ref{a:singularity}. Then the resolvents $(\Delta^{\pm}_{\Omega}-z)^{-1}$ of the Dirichlet or Neumann Laplace-Beltrami operators are compact for any $z>0$, where $\Omega = \M_a^b$  for  $0\leq a< b \leq \infty$.
\end{prop}

In particular, the spectrum of $\Delta^{\pm}_{\Omega}$ is discrete. We denote by $N_{[a,b]}^\pm(\lambda)$ the corresponding Weyl's counting functions. The following instance of Dirichlet-Neumann bracketing holds  as a consequence of the min-max principle (see \cite[p.\ 407]{CH-methods}).

\begin{prop}[Dirichlet-Neumann bracketing]\label{p:d-n}
 Let $\M$ be a Riemannian manifold with compact metric completion and satisfying Assumption \ref{a:singularity}. Then, for any  sequence $0=a_0<a_1<\ldots<a_{m+1}=\infty$, we have
  \begin{equation}
    \sum_{i=0}^m N^-_{[a_i,a_{i+1}]}(\lambda)\le N(\lambda)\le \sum_{i=0}^m N^+_{[a_i,a_{i+1}]}(\lambda) , \qquad \forall \lambda\ge0.
  \end{equation}
\end{prop}

In order to discard the contributions to $N(\lambda)$ of the regions near the metric boundary, we need the following Lemma. It is an immediate consequence of the min-max principle and the Hardy inequality given by Proposition~\ref{p:hardy-appendix}.

\begin{lemma}[Estimates close to the singularity]\label{lem:hardy}
 Let $\M$ be a Riemannian manifold with compact metric completion and satisfying Assumption \ref{a:singularity}.
 Then, for $0<\varepsilon<\varepsilon_0 / 2$, it holds
  \begin{equation}
     N^{\pm}_{[0,\varepsilon]}(\lambda)=0, \qquad \forall\lambda<\frac{1}{4\varepsilon^{2}}.
  \end{equation}
\end{lemma}

We can now prove Theorem~\ref{t:weyl}. The argument consists in the following steps:
\begin{enumerate}
\item Apply Proposition \ref{p:d-n} to the decomposition $\M_0^\varepsilon \cup \M^\infty_\varepsilon$, for small $\varepsilon>0$.
\item Use Theorem~\ref{t:weylremainder} to evaluate $N^\pm_{[\varepsilon,\infty]}(\lambda)$ with an explicit remainder term.
\item Relate $\varepsilon$ to $\lambda$ in such a way that the contribution of $N_{[0,\varepsilon(\lambda)]}(\lambda)$ is negligible, thanks to Lemma~\ref{lem:hardy}, and the remainder term in $N_{[\varepsilon(\lambda),\infty]}(\lambda)$ is controlled as $\lambda\to\infty$.
\end{enumerate}

\begin{proof}[Proof of Theorem~\ref{t:weyl}]
Let $0<\varepsilon  < \varepsilon_0/2$. We split $\M$ into two parts $\M= \M_0^\varepsilon \cup \M_\varepsilon^\infty$.  By Lemma \ref{l:basic} we can apply Theorem~\ref{t:weylremainder} to $\M_\varepsilon^\infty$. Furthermore, $\lambda_0$ given in \eqref{eq:lambda0} satisfies $\lambda_0 \leq  b/\varepsilon^2$, for some constant $b$. It depends on $n$ and on the constants appearing in Assumption \ref{a:singularity}, but not on $\epsilon$. For simplicity, we set $b=1$ in the following. Thus there exists $C>0$ such that for all $\epsilon<\epsilon_0/2$ it holds
\begin{equation}\label{eq:weyl-admissible}
	\left\lvert \frac{(2\pi)^n}{\omega_n \vol(\M_\varepsilon^\infty)\lambda^{n/2}}N_{[\varepsilon,\infty]}^{\pm}(\lambda) -1 \right\rvert \leq \frac{C}{\log(1+\sqrt{\lambda\epsilon^2})}, \qquad \forall \lambda >0.
\end{equation}
Our aim is to let $\varepsilon \to 0$ as $\lambda \to \infty$, while keeping $\varepsilon^2\lambda$ bounded away from zero in order to keep the remainder term under control.
 Hence, let $a$ be a positive constant and set
\begin{equation}
\varepsilon_a(\lambda) := \frac{1}{\sqrt{a\lambda}}, \qquad a>0.
\end{equation}
In this case, setting $\epsilon = \varepsilon_a(\lambda)$, the remainder term in \eqref{eq:weyl-admissible} is bounded by a constant, depending only on the dimension and on $a$, which can be made arbitrarily small as $a \to 0$.

In the rest of the proof, by considering two cases in which $a$ is either large or small, we obtain the upper and lower bound for $N(\lambda)$, respectively.

We start with the upper bound. Choose $a_+ > 4$ and set $\varepsilon = \varepsilon_{a_+}(\lambda)$ as described above. Then, Lemma~\ref{lem:hardy} yields $N_{[0,\varepsilon]}^{\pm}(\lambda)=0$ for $\lambda\ge \Lambda_+:=4/(\varepsilon_0^2a_+)$. Hence, by Neumann bracketing (i.e., the r.h.s.\ of Proposition~\ref{p:d-n}) we obtain that there exists $C_+>0$ such that for all $\lambda\ge \Lambda_+$ it holds
\begin{equation}\label{eq:ub}
N(\lambda) \leq N_{[0,\varepsilon]}^{+}(\lambda) + N_{[\varepsilon,\infty]}^{+}(\lambda) \leq N_{[\varepsilon,\infty]}^{+}(\lambda) \leq C_+\lambda^{n/2} \vol\left(\M_{1/\sqrt{a_{+}\lambda}}^\infty\right).
\end{equation} 

For the lower bound, we neglect the boundary contribution, since $N^+_{[0,\varepsilon]}(\lambda) \geq 0$. By Dirichlet bracketing (i.e., the l.h.s.\ of Proposition~\ref{p:d-n}), we have
\begin{equation}
N(\lambda) \geq N_{[0,{\varepsilon}]}^{-} + N_{[\varepsilon,\infty]}^{-}(\lambda) \geq  N_{[\varepsilon,\infty]}^{-}(\lambda).
\end{equation}
Choose $\varepsilon = \varepsilon_{a_-}(\lambda)$, with $a_-$ sufficiently small in such a way that the remainder term in \eqref{eq:weyl-admissible} is smaller than $1$. We deduce that there exists a constant $C_->0$ such that 
\begin{equation}\label{eq:lb}
  N(\lambda) \geq  N_{[\varepsilon,\infty]}^{-}(\lambda) \ge C_-\lambda^{n/2} \vol\left(\M_{1/\sqrt{a_-\lambda}}^\infty\right),
\end{equation}
provided that $\lambda \geq \Lambda_- := 4/(\varepsilon_0^2a_-)$. 

To conclude the proof,  we apply Lemma~\ref{l:vol} to \eqref{eq:ub} and \eqref{eq:lb}.
\end{proof}
\begin{rmk}\label{r:noimprov}
The proof of Theorem~\ref{t:weyl} shows that the estimate of Theorem~\ref{t:traceheat} (and in turn Theorem~\ref{t:weyl} itself) cannot be improved. Indeed, suppose that we are able to deduce a better remainder term, so that by setting $\varepsilon(\lambda):=(a\lambda)^{-1/2}$
%\begin{equation}
%\varepsilon = \frac{1}{\sqrt{a\lambda}},
%\end{equation}
the remainder term of Theorem~\ref{t:traceheat} is negligible, and not simply bounded, as $\lambda \to \infty$. For the upper bound, arguing as above, we need to choose $a = a_+ >4$, and we obtain
\begin{equation}\label{eq:bestupper}
\limsup_{\lambda \to \infty} \frac{N(\lambda)}{\lambda^{n/2}\vol\left(\M^\infty_{1/\sqrt{\lambda}}\right)}  \leq \frac{\omega_n}{(2\pi)^n} a_+^{1/2C},
\end{equation}
where we used Lemma \ref{l:vol} to derive that $\vol(\M^\infty_{1/\sqrt{a_+\lambda}})/\vol(\M^\infty_{1/\sqrt{\lambda}}) \leq  a_+^{1/2C}$.  Hence, the best upper bound in \eqref{eq:bestupper} is obtained for $a_+ = 4$. For the lower bound, we obtain
\begin{equation}
\liminf_{\lambda \to \infty} \frac{N(\lambda)}{\lambda^{n/2}\vol\left(\M^\infty_{1/\sqrt{\lambda}}\right)} \geq \frac{\omega_n}{(2\pi)^n} a_-^{1/2C}.
\end{equation}
In this case, there is no constraint on $a_-$, obtaining a contradiction with \eqref{eq:bestupper}.
\end{rmk}

\section{Slowly varying volumes}\label{s:exact}

A measurable function $\ell:\mathbb R_+\to \mathbb R_+$ is slowly varying at infinity in the sense of Karamata (cf. \cite{regularvariation}) if, for all $a >0$, it holds
\begin{equation}\label{eq:regular}
  \lim_{x\to \infty} \frac{\ell(ax)}{\ell(x)} = 1.
\end{equation}

\begin{example}\label{eq:examplesslowly}
Examples of slowly varying functions, cf.\ \cite{regularvariation}, are $\log x$, the iterates $\log_k x = \log_{k-1}\log x$, rational functions with positive coefficients formed with the $\log_k x$. Non-logarithmic examples are
\begin{equation}\label{eq:non-log}
\exp\left((\log x)^{\alpha_1} \dots (\log_k x)^{\alpha_k} \right), \qquad 0<\alpha_i<1.
\end{equation}
Clearly, any function with finite limit at infinity is slowly varying.
\end{example}

\subsection{Exact Weyl's law for slowly varying volumes}

{The main result of this section is an \emph{exact} Weyl's law for singular structures satisfying Assumption \ref{a:singularity} and an additional volume growth assumption.}
\begin{theorem}\label{t:weylexact}
  Let $\M$ be an $n$-dimensional Riemannian manifold with compact metric completion and satisfying Assumption~\ref{a:singularity}. Assume, moreover, that the function
\begin{equation}
  \vv(\lambda)=\vol\left(\M^\infty_{1/\sqrt{\lambda}}\right)
\end{equation}
is slowly varying. Then, we have
\begin{equation}
        \lim_{\lambda\to \infty} \frac{N(\lambda)}{\lambda^{n/2}\vv(\lambda)} = \frac{\omega_n}{(2\pi)^n}.
    \end{equation}
\end{theorem}
\begin{proof}
We prove that
\begin{equation}
\frac{\omega_n}{(2\pi)^n}\leq \liminf_{\lambda \to \infty} \frac{  N(\lambda)}{\lambda^{n/2} \vv(\lambda)}  \leq \limsup _{\lambda \to \infty} \frac{  N(\lambda)}{\lambda^{n/2} \vv(\lambda)}  \leq \frac{\omega_n}{(2\pi)^n} .
\end{equation}
The proof of the lower bound starts as in the proof of Theorem~\ref{t:weyl}, i.e., by splitting $\M =\M_\varepsilon^\infty \cup \M_0^\varepsilon$, for $0<\varepsilon<\varepsilon_0/2$. We have
\begin{equation}
N(\lambda) \geq N_{[0,{\varepsilon}]}^{-} + N_{[\varepsilon,\infty]}^{-}(\lambda) \geq  N_{[\varepsilon,\infty]}^{-}(\lambda).
\end{equation}
Choose $a >0$ small and let $\varepsilon = 1/\sqrt{a\lambda}$. From \eqref{eq:weyl-admissible} we deduce the existence of a constant $C(a)  = - C/\log(1+1/a)$, tending to $0$ as $a \to 0$, such that
\begin{equation}
  N(\lambda) \geq  N_{[\varepsilon,\infty]}^{-}(\lambda) \ge \frac{\omega_n}{(2\pi)^n} \lambda^{n/2} \vv(a \lambda) (1 + C(a)).
\end{equation}
We now use the fact that $\vv(\lambda)$ is slowly varying to obtain
\begin{equation}
\liminf_{\lambda \to \infty} \frac{  N(\lambda)}{\lambda^{n/2} \vv(\lambda)} \geq \frac{\omega_n}{(2\pi)^n}(1+C(a)).
\end{equation}
By letting $a \to 0$ we conclude the proof of the lower bound.

{ The upper bound is more delicate.} We split $\M$ into three parts:
\begin{equation}
\M = \M_{0}^{\varepsilon_1} \cup \M_{\varepsilon_1}^{\varepsilon_2} \cup \M_{\varepsilon_2}^\infty, 
\qquad 0<\varepsilon_1<\varepsilon_2 < \varepsilon_0/2.
\end{equation}
Consider $a<1$ small and let
\begin{equation}\label{eq:choiceeps1eps2}
  \varepsilon_1(\lambda) := \frac{1}{10\sqrt{\lambda}}, \qquad 
  \varepsilon_2(\lambda) := \frac{1}{\sqrt{a\lambda}}.
\end{equation}
The factor $10$ above has been chosen in order to be able to apply Lemma \ref{lem:hardy}, whence $N_{[0,\varepsilon_1(\lambda)]}^{-}(\lambda)=0$. By Neumann bracketing, we obtain
\begin{equation}\label{eq:1Neumann}
N(\lambda) \leq 
N_{[\varepsilon_1(\lambda),\varepsilon_2(\lambda)]}^{+}(\lambda) +N_{[\varepsilon_2(\lambda),\infty]}^{+}(\lambda).
\end{equation}

In Proposition \ref{p:intermediatecouche} below, we show that, thanks to the slowly varying assumption, the first term in \eqref{eq:1Neumann} gives a negligible contribution at infinity, more precisely
\begin{equation}\label{eq:contr1}
\lim_{\lambda \to \infty} \frac{N_{[\varepsilon_1(\lambda),\varepsilon_2(\lambda)]}^{+}(\lambda)}{\lambda^{n/2}\vv(\lambda)} = 0.
\end{equation}
On the other hand, applying Theorem~\ref{t:weylremainder} to $\M_{\varepsilon_2(\lambda)}^\infty$, we obtain that for all $\lambda >0$
\begin{equation}\label{eq:contr2}
N_{[\varepsilon_2(\lambda),\infty]}^{+}(\lambda)  \leq \frac{\omega_n}{(2\pi)^n} \vol\left(\M_{\varepsilon_2(\lambda)}^{\infty}\right) \lambda^{n/2} \left(1+C(a)\right),
\end{equation}
where $C(a)\to 0$ as $a \to 0$.
Since $\vv$ is slowly varying, we have
\begin{equation}
%  \vol\left(\M_{\varepsilon_1(\lambda)}^{\varepsilon_2(\lambda)}\right) = o\left(\vv(\lambda)\right)
%  \qquad\text{and}\qquad 
  \vol\left(\M_{\varepsilon_2(\lambda)}^{\infty}\right) \sim \vv(\lambda).
\end{equation}
Putting together the contributions from \eqref{eq:contr1} and \eqref{eq:contr2}, we finally get
\begin{equation}
  \limsup_{\lambda\to\infty}\frac{N(\lambda)}{\lambda^{n/2}\vv(\lambda)}\le \frac{\omega_n}{(2\pi)^n}(1+C(a)).
\end{equation}
Letting $a \to 0$, we have $C(a) \to 0$, which completes the proof.
\end{proof}

  The following proposition estimates the number of eigenvalues in the intermediate strip $\M_{\varepsilon_1}^{\varepsilon_2}$ close to the singularity. Note that we cannot apply Theorem~\ref{t:weylremainder} to $\M_{\varepsilon_1}^{\varepsilon_2}$ since the latter does not have convex boundary.
  \begin{prop}\label{p:intermediatecouche}
     Let $\M$ be an $n$-dimensional Riemannian manifold with compact metric completion and satisfying Assumption \ref{a:singularity}. There exists a constant $C>0$ such that, for any $0<\varepsilon_1<\varepsilon_2 < \varepsilon_0 /2$, it holds
    \begin{equation}
    N^{\pm}_{[\varepsilon_1,\varepsilon_2]}(\lambda) \leq C \vol(Z_{\varepsilon_1})(\varepsilon_2-\varepsilon_1) \left( \frac{\varepsilon_2}{\varepsilon_1}\right)^{C/2}\lambda^{n/2}, \qquad \forall \lambda >\frac{(\varepsilon_1/\varepsilon_2)^C}{\min\{\varepsilon_1^2,(\varepsilon_2-\varepsilon_1)^2\}}.
    \end{equation}
    Assume furthermore that $\vv(\lambda)$ is slowly varying, and choose $\varepsilon_1 = \frac{1}{10\sqrt{\lambda}}$ and $\varepsilon_2 = \frac{1}{\sqrt{a\lambda}}$ as in \eqref{eq:choiceeps1eps2}, with $a <1$ sufficiently small. Then we have
\begin{equation}
  \lim_{\lambda \to \infty}\frac{N^{\pm}_{[\varepsilon_1(\lambda),\varepsilon_2(\lambda)]}(\lambda)}{\lambda^{n/2}\vv(\lambda)} = 0  .
  \end{equation}
  \end{prop}  
\begin{proof}
  It is sufficient to prove the proposition for the Neumann case.
Let $I=[\varepsilon_1,\varepsilon_2]$. Close to the metric boundary, one has that $\M_{\varepsilon_1}^{\varepsilon_2} = I \times Z$ and
\begin{equation} 
g = dx^2 + h(x),
\end{equation}
where $h(x)$ is a one-parameter family of Riemannian metrics on the fixed closed hypersurface $Z$.  Let $Q$ and $R$ be the corresponding quadratic form and Rayleigh quotient, i.e.,
  \begin{equation}
    Q(u) = \int_{ I \times Z} |\nabla^g u|^2\,d\mu_g, \qquad R(u) = \frac{Q(u)}{\|u\|^2_{L^2(I\times Z, d\mu_g)}}, \qquad u \in C^\infty(I\times Z).
  \end{equation}
The idea is to control the Rayleigh quotient in terms of the one of a simpler metric. To this purpose, let  $g_1$ be the metric on $I \times Z$ obtained by freezing $x = \varepsilon_1$, that is
\begin{equation}
g_1 = dx^2 + h(\varepsilon_1).
\end{equation}
Fix a smooth measure $dz$ on $Z$. Observe that $d\mu_g = e^{2\theta(x,z)} dx dz$ for a smooth function $\theta : I \times Z \to \R$, and that $\tr \Hess(\delta) = 2\partial_x \theta$. Therefore, since $\delta(x,z)= x$ on $I\times Z$, we have
\begin{equation}
-\frac{C}{x} \leq 2\partial_x\theta \leq 0,
\end{equation}
for some constant $C>0$ depending only on $n$. It follows that on $I\times Z$ it holds
\begin{equation}\label{eq:ineqmeasures}
\left(\frac{\varepsilon_1}{\varepsilon_2}\right)^C  d\mu_{g_1} \leq d\mu_g \leq d\mu_{g_1},
\end{equation}
as measures. Inequality \eqref{eq:ineqmeasures} will be used to estimate the behaviour of the measure in the Rayleigh quotient. For what concerns the norm of the gradient, observe that, by convexity, the family $x\mapsto h(x)$ is decreasing, which implies 
\begin{equation}\label{eq:inequalitynorms}
|\nabla^{g} u|^2  \ge |\nabla^{g_1} u |^2.
\end{equation}

It follows from \eqref{eq:ineqmeasures} and \eqref{eq:inequalitynorms} that, denoting with $R_1$ the Rayleigh quotient of the Riemannian manifold $(I\times Z, g_1)$, one has
\begin{equation}
R(u) \geq \left(\frac{\varepsilon_1}{\varepsilon_2}\right)^C R_1(u), \qquad \forall u \in C^\infty(I\times Z).
\end{equation}
By the min-max characterization of eigenvalues, it follows that
\begin{equation}\label{eq:splittamento}
N^+_{(I \times Z,g)}(\lambda) \leq N^+_{(I \times Z, g_1)}\left(\left(\frac{\varepsilon_2}{\varepsilon_1}\right)^C \lambda\right), \qquad \forall \lambda >0.
\end{equation}
We will estimate the r.h.s.\ of \eqref{eq:splittamento} through Theorem \ref{t:weylremainder}. To do so, notice that $(I\times Z,g_1)$ is the product of $(I,dx^2)$ and $(Z,h_1)$, with $h_1 := h(\varepsilon_1)$. As such, it is a compact Riemannian manifold with totally geodesic, and thus convex, boundary. Its sectional curvature is bounded by the one of the factor $(Z,h_1)$. By Gauss' equation, there exists a constant $C$ (depending only on the constants appearing in Assumption \ref{a:singularity} and hence not on the choice of $\varepsilon_1$ and $\varepsilon_2$) such that
\begin{equation}
|\Sec(I \times Z, g_1)| \leq |\Sec(Z,h_1)| \leq \frac{C}{\varepsilon_1^2}.
\end{equation}
Furthermore, the injectivity radius from the boundary of $(I \times Z,g_1)$ is
\begin{equation}
\inj_{\partial}(I\times Z,g_1) = \frac{1}{2}(\varepsilon_2-\varepsilon_1).
\end{equation}
Finally, the injectivity radius of $(I\times Z,g_1)$ is equal to the one of $(Z,h_1)$. The latter is a submanifold of bounded second fundamental form in a Riemannian manifold of bounded sectional curvature and injectivity radius. Its injectivity radius can be bounded from below in terms of the aforementioned quantities and its distance from the metric boundary, as stated in Lemma \ref{l:injsubmanifold}. Using Assumption \ref{a:singularity}, we deduce the existence of $C>0$, not depending on the choice of $\varepsilon_1,\varepsilon_2$, such that
\begin{equation}
\inj(I\times Z, g_1) = \inj(Z,h_1) \geq C^{-1} \varepsilon_1.
\end{equation}
We can now apply Theorem \ref{t:weylremainder} to $(I\times Z,g_1)$, yielding the existence of a constant $C>0$, not depending on the choice of $\varepsilon_1,\varepsilon_2$, such that
\begin{equation}\label{eq:eqfinalll}
N^+_{[\varepsilon_1,\varepsilon_2]}(\lambda) \leq C  \vol(I\times Z,g_1) \left(\left(\frac{\varepsilon_2}{\varepsilon_1}\right)^{C}\lambda\right)^{n/2}, \qquad \forall \lambda > \frac{(\varepsilon_1/\varepsilon_2)^C}{\min\{\varepsilon_1^2,(\varepsilon_2-\varepsilon_1)^2\}}.
\end{equation}
This proves the first part of the proposition, as $\vol(I\times Z, g_1) = \vol(Z_{\varepsilon_1})(\varepsilon_2-\varepsilon_1)$. 

To prove the second part of the statement, letting $Z_x=\{\delta =x\}$ and recalling the definition $\vv(1/x^2) = \vol(\M_x^\infty)$, we deduce 
\begin{equation}
\vol(\M_x^\infty) = \int_x^\infty \vol(Z_x)\, dx \qquad \Rightarrow \qquad \vol(Z_{\varepsilon_1}) = \frac{\vv'(1/\varepsilon_1^2)}{2 \varepsilon_1^3}.
\end{equation}
Let now choose $\varepsilon_1 = \frac{1}{10\sqrt{\lambda}}$ and $\varepsilon_2 = \frac{1}{\sqrt{a \lambda}}$, for $a<1$. Notice that for $a$ sufficiently small, depending on the given value of $C$, then the condition for the validity of \eqref{eq:eqfinalll} is verified for all $\lambda$. We have, in this case, renaming the constants (which may now depend on $a$),
\begin{equation}
\frac{N_{[\varepsilon_1(\lambda),\varepsilon_2(\lambda)]}^+(\lambda)}{\lambda^{n/2} \vv(\lambda)} \leq C(a) \frac{\lambda \vv'(100\lambda)}{\vv(\lambda)} , \qquad \forall \lambda >0.
\end{equation}
Since $\vv$ is slowly varying the r.h.s.\ tends to zero (use Lemma~\ref{lem:lamperti} given below).
\end{proof}

\noindent The next result is an application of \cite[Thm.\ 1.7.2 and Prop.\ 1.5.8]{regularvariation}, see also \cite[Thm.\ 2]{Lamperti1958}.
\begin{lemma}\label{lem:lamperti}
  Let $\vv:\R_+\to \R_+$ be a slowly varying function of class $C^1$ such that $\lambda\mapsto\lambda^a \vv'(\lambda)$ is monotone for some $a\geq 0$. Then, 
  \begin{equation}\label{eq:lamperti}
    \lim_{\lambda\to \infty} \frac{\lambda \vv'(\lambda)}{\vv(\lambda)} = 0.
  \end{equation}
\end{lemma}

  \begin{rmk}
    A natural relaxation of the slowly varying condition in Theorem~\ref{t:weylexact}, is to require on $\vv$ to be $k$-regularly varying for some $k
    >0$. Namely,
    \begin{equation}
      \lim_{\lambda\to \infty}\frac{\vv(a\lambda)}{\vv(\lambda)} = a^{k}, \qquad \forall a>0.
    \end{equation}
    Indeed, $k=0$ corresponds to the slowly varying case. %  If $\vv$ if $k$-regularly varying, then $\vv(\lambda) = \lambda^k \tilde\vv(\lambda)$ for some slowly varying function $\tilde \vv$. 
    The  same argument of proof of Theorem~\ref{t:weylexact} yields
\begin{equation}
a^k(1-C(a)) \leq \liminf_{\lambda \to \infty} \frac{N(\lambda)}{\tfrac{\omega_n}{(2\pi)^n}\lambda^{n/2}\vv(\lambda)} \leq  \limsup_{\lambda \to \infty} \frac{N(\lambda)}{\tfrac{\omega_n}{(2\pi)^n}\lambda^{n/2}\vv(\lambda)}  \leq a^k(1+C(a)) + \frac{k}{C(a)},
\end{equation}
where $C(a)>0$ is a constants such that $C(a) \to 0$ as $a \to 0$. One can see that, unless $k=0$ (i.e.\ the case of slowly varying $\vv$), one obtains a result that is equivalent to Theorem \ref{t:weyl}. We do not know whether this is a limit of our approach or, on the contrary, there is no exact Weyl's law in the case of regularly varying volume functions.
  \end{rmk}

\subsection{Metrics with prescribed Weyl's law}\label{s:prescribed}

We prove the following converse to Theorem \ref{t:weylexact}.

\begin{theorem}\label{t:converse}
Let $\KK$ be an $n$-dimensional compact manifold, $S\subset \KK$ be a closed submanifold, and $\vv : \R_+ \to \R_+$ be a non-decreasing slowly varying function.  Then,  there exists a Riemannian structure on $\KK$, singular at $S$, such that
\begin{equation}\label{e:exact-prescribed}
\lim_{\lambda \to \infty} \frac{N(\lambda)}{\lambda^{n/2} \vv(\lambda)} = \frac{\omega_n}{(2\pi)^n}.
\end{equation}
\end{theorem}
\begin{proof}
The idea is to build a non complete Riemannian structure on $\KK \setminus S$, of warped-product type near $S$, with respect to some function $f$, which has to be chosen carefully, so that $\vol(M_{1/\sqrt{\lambda}})\sim \vv(\lambda)$ and Assumption~\ref{a:singularity} is satisfied. Indeed, this will allow to apply Theorem~\ref{t:weylexact}, and thus to obtain \eqref{e:exact-prescribed}. To this purpose, one needs to control in a precise way the asymptotic behaviour of the quantities $\lambda \vv''(\lambda)/\vv'(\lambda)$ and $\lambda^2 \vv^{(3)}(\lambda)/v'(\lambda)$; However, this is not possible for general slowly varying functions.\footnote{For all examples of monotone slowly varying function given at the beginning of this section, these quantities actually admit a finite limit. An example of strictly monotone slowly varying function for which $\lambda \vv''(\lambda)/\vv'(\lambda)$ is unbounded and does not have a limit is $\vv(\lambda) = 2\log\lambda+\sin\log\lambda$.} We tackle this problem by exploiting the theory of regular variation to replace $\vv$ with a more tame slowly varying function with the desired asymptotics at infinity. 

We recall from \cite[Ch.\ 3]{regularvariation} the following definitions. The de Haan class $\Pi$ is the set of those measurable functions $\phi : \R_+\to\R_+$ for which there exists a slowly varying function $\ell:\R_+\to \R_+$, and $c>0$ such that 
\begin{equation}
\lim_{\lambda \to \infty} \frac{\phi(a\lambda)-\phi(\lambda)}{\ell(\lambda)} = c\log a, \qquad \forall a>0.
\end{equation}
One can see that $\Pi$ is a (strict) subset of slowly varying functions \cite[p.\ 164]{regularvariation}. The smooth de Hann class $S\Pi$ is the set of those smooth functions $\phi : \R_+\to\R_+$ that satisfy
\begin{equation}\label{eq:SP}
\lim_{\lambda \to \infty} \frac{\lambda^{m} \phi^{(1+m)}(\lambda)}{\phi^{'}(\lambda)} = (-1)^m m!, \qquad \forall m \in \N.
\end{equation}
It holds $S\Pi \subset \Pi$, \cite[p.\ 165]{regularvariation}.

By \cite[Appendix B]{deHaan}, any non-decreasing slowly varying function $\vv$ is asymptotic to a de Haan function. Furthermore, by a smoothing result \cite[Thm.\ 3.7.7]{regularvariation}, any de Haan function is asymptotic to a smooth de Haan function, which we denote with the same symbol $\vv \in S\Pi$. Moreover, it follows from the proof of \cite[Thm.\ 3.7.7]{regularvariation} that $\vv'(\lambda)>0$ for large $\lambda$. Thus, we assume that $\vv$ is smooth, strictly increasing, and satisfies \eqref{eq:SP}.

We proceed with the construction in the case where $S$ is a submanifold of codimension $\neq 1$ or it is one-sided. The case of a two-sided hypersurface follows with trivial modifications. Choose a tubular neighborhood $\mathcal{O}\subset \KK$ of $S$ such that $\mathcal{O}\setminus S = (0,2) \times Z$, for a closed hypersurface $Z$. Fix a metric $\hat{g}$ on $Z$ and set
\begin{equation}\label{eq:warped}
g|_{\mathcal{O}} = dx^2 + f^2 \hat{g},
\end{equation}
where $f : (0,2) \to \R_+$ is a smooth function to be chosen later and meant to explode as $x$ tends to $0$. Extend $g$ to a smooth Riemannian metric on the whole $\M:=\KK \setminus S$, by preserving \eqref{eq:warped} on the neighborhood $(0,1) \times Z$.

By construction, $(\M,g)$ has compact metric completion and $\distmb(x,z)= x$ for $(x,z) \in (0,1)\times Z$. The level sets of $\distmb$ close to the metric boundary are diffeomorphic to $Z$ and (a) of Assumption \ref{a:singularity} is verified. Define $f$ in such a way that $\vol(M_{1/\sqrt{\lambda}}) \sim \vv(\lambda)$, by setting
\begin{equation}\label{eq:choicef}
f(x)^{n-1} :=  \frac{2}{\vol(Z,\hat{g})} \frac{\vv'(1/x^2)}{ x^{3}}, \qquad x\ \in (0,1).
\end{equation}
Since $\vv$ is strictly increasing, $f>0$ and its only singularity is at $x=0$.

Let us verify that $(\M,g)$ satisfies Assumption~\ref{a:singularity}. The projection on the first factor $\pi : (0,1) \times Z \to (0,1)$ of the warped product \eqref{eq:warped} is a Riemannian submersion with leaves $(Z,\hat{g})$. By O'Neill formulas \cite[9.29, 9.104]{Besse}, the sectional curvatures are:
%\begin{align}
%K(U,V) & = \frac{1}{f^2}\hat{K}(U,V)- \left(\frac{f'}{f}\right)^2, \\
%K(X,U) & = - \frac{f''}{f}, \\
%K(X,Y) & = 0.
%\end{align}
\begin{align}
K(U,V) = \frac{1}{f^2}\hat{K}(U,V)- \left(\frac{f'}{f}\right)^2, \qquad
K(X,U)  = - \frac{f''}{f}, \qquad
K(X,Y) = 0.
\end{align}
Here, $U,V$ are orthonormal vectors tangent to the fibers $Z$ of the submersion, $X, Y$ are unit vectors tangent to the base $(0,1)$, and $\hat{K}$ is the sectional curvature of $(Z,\hat{g})$. 
Finally, the hypersurfaces $\{\distmb = x\}$, for $x\in (0,1)$, have as their second fundamental form
\begin{equation}
\Hess(\delta)(U,V) = %\mathrm{sgn}(x)
\frac{f'}{f} g(U,V).
\end{equation}
It is then clear that the quantities controlling the behavior of the geometric invariants of $g$ close to the metric boundary (i.e.\ as $x \to 0$) are $f'/f$ and $f''/f$. Thanks to the fact that $\vv \in S\Pi$, we are able to compute their asymptotics as $x\to 0$. To this purpose, set
\begin{equation}
\xi_m(x):= \frac{\vv^{(m+1)}(1/x^2)}{x^{2 m} \vv'(1/x^2)}, \qquad m\geq 1.
\end{equation}
Using \eqref{eq:SP}, we have $\xi_m(x) \to (-1)^m m!$ as $x \to 0$. Thus, using \eqref{eq:choicef}, we have, as $x \to 0$,
\begin{align}
\frac{f'(x)}{f(x)} & = - \frac{3 + 2\xi_1(x)}{(n-1) x} \sim - \frac{1}{(n-1) x}, \\
\frac{f''(x)}{f(x)} & = \frac{3n+6 -4(n-2)\xi_1(x)^2 + 4(n-1)\xi_2(x) + 6(n+1)\xi_1(x)}{(n-1)^2 x^2} \sim \frac{n}{(n-1)^2 x^2}.
\end{align}
Hence, Assumption~\ref{a:singularity} is verified and we can apply Theorem \ref{t:weylexact} to $(\M,g)$.
\end{proof}

\begin{rmk}\label{r:ess-adj}
The Laplace-Beltrami operator of the structure built in the proof of Theorem \ref{t:converse}, with domain $C^\infty_c(\KK\setminus S)$, is essentially self-adjoint in $L^2(\KK,d\mu_g)$. Indeed, $\Delta \distmb = (n-1)f' /f$, and as a consequence of the curvature estimates obtained above one has
\begin{equation}
V_{\mathrm{eff}} := \left( \frac{\Delta \distmb}{2}\right)^2 + \left(\frac{\Delta \distmb}{2}\right)^\prime \geq \frac{3}{4\distmb^2}\left(1- \frac{1}{\log \distmb^{-1}}\right),
\end{equation}
for small $\delta$, where $^\prime$ denotes the derivative in the direction of $\nabla\delta$. This estimate allows to apply the essential-self adjointness criterion of \cite{QC}, combined with the improvement of the constant obtained in \cite{Nenciu}. We omit the details. 
\end{rmk}

\section{Concentration of eigenfunctions}
\label{sec:concentration}

Under the assumptions of Theorem \ref{t:weyl}, it holds that $N(\lambda)\asymp \lambda^{n/2} \vol(\M_{1/\sqrt{\lambda}}^\infty)$. Here, $f(\lambda)\asymp g(\lambda)$ means that the ratio $f(\lambda)/g(\lambda)$ is uniformly bounded above and below by positive constants for $\lambda$ large enough. 
In this section, we show that under the additional assumption $\vol(\M) = \infty$, eigenfunctions concentrate at the metric boundary of $\M$. 

We recall that a subset $\SSS \subseteq \N$ has density $a \in [0,1]$ if
\begin{equation}
\lim_{{\ell} \to \infty} \frac{1}{{\ell}} \sum_{k=0}^{{\ell}-1} \mathbbold{1}_{\SSS}(k) =a.
\end{equation}

\begin{theorem}\label{t:localization}
Let $\M$ be an $n$-dimensional Riemannian manifold such that the Laplace-Beltrami operator $\Delta$ has discrete spectrum, and
\begin{equation}
\lim_{\lambda \to \infty} \frac{N(\lambda)}{\lambda^{n/2}} = \infty.
\end{equation}
Let $\{\phi_i\}_{i\in \N}$, be a complete set of normalized eigenfunctions of $-\Delta$, associated with eigenvalues $\lambda_i$, arranged in non-decreasing order. Then, there exists a density one subset $\SSS \subseteq \N$ such that for any compact $U$ it holds
\begin{equation}
\lim_{\substack{i \to \infty \\ i \in \SSS}} \int_U |\phi_i|^2d\mu_g = 0.
\end{equation}
\end{theorem}
\begin{proof}
Fix a compact set $U$. Let $E$ denote the heat kernel of $\Delta$ and, for $i\in \N$, $a_{i}(U) := \int_U  |\phi_i|^2d\mu_g$. Since the heat kernel is local, we have
\begin{equation}
t^{n/2} \sum_{i=1}^\infty e^{-t\lambda_i} a_{i}(U) = t^{n/2} \int_U E(t,q,q)d\mu_g(q)\sim c,{ \qquad t\to 0,}
\end{equation}
for some constant $c>0$. By the standard Karamata theorem, it holds
\begin{equation}
\sum_{\lambda_i \leq \lambda} a_{i}(U) \sim \frac{c}{\Gamma(n/2+1)} \lambda^{n/2}, \qquad \lambda \to \infty.
\end{equation}
By our assumption on $N(\lambda)$, it holds then
%\begin{equation}
%\lim_{\lambda \to \infty} \frac{1}{N(\lambda)}\sum_{\lambda_i < \lambda} a_{i}(U) = 0,
%\end{equation}
%or, equivalently
\begin{equation}
\lim_{{\ell} \to \infty} \frac{1}{{\ell}} \sum_{i=1}^{{\ell}-1} a_{i}(U) = 0.
\end{equation}
By \cite[Lemma 6.2]{Ergodic}, the above statement is equivalent to the existence of a density one {subset $\SSS_U \subseteq \N$} such that
\begin{equation}
\lim_{\substack{i \to \infty \\ i \in \SSS_U}} a_{i}(U) = 0.
\end{equation}
The subset $\SSS_U \subset \N$ depends on the choice of $U$ but we next build a subset $\SSS$ having the same property and which does not depend on $U$, as claimed in the statement. We use ideas similar to those in the proof of \cite[Lemma 6.2]{Ergodic} and \cite[Sec.\ 5]{CdVergodicity}.

Let $\{U_m\}_{m \in \N}$ be an exhaustion of $\M$, that is each $U_m$ is compact, $U_{m+1} \supset U_{m}$, and $U_m \to \M$ as $m \to \infty$. Let $\SSS_m \subset \N$ a density one subset built as above, such that
\begin{equation}\label{eq:propm}
\lim_{\substack{i \to \infty \\ i \in \SSS_m}} a_i(U_m) = 0.
\end{equation}
Without loss of generality, we can assume that $\SSS_{m+1} \subseteq \SSS_m$ (if this is not the case, we take in place of $\SSS_m$ the set $\tilde{\SSS}_m = \cap_{i \leq m} \SSS_i$. Indeed $\tilde{\SSS}_m$ is a density-one subset of $\N$ with the required properties, and such that \eqref{eq:propm} holds). By the density one property, there exists $i_1<i_2<\dots$ such that
\begin{equation}
\frac{1}{{\ell}} \sum_{k=0}^{{\ell}-1} \mathbbold{1}_{\SSS_m}(k) \geq 1 - \frac{1}{m}, \qquad \forall {\ell} \geq i_{m-1}.
\end{equation}
Then, the required set $\SSS$ can be taken as
\begin{equation}
\SSS := \bigcup_{m=1}^\infty \SSS_m \cap [i_{m},i_{m+1}),
\end{equation}
Indeed, if $i_{m} \leq n < i_{m+1}$ we have
\begin{align}
\frac{1}{{\ell}} \sum_{k=0}^{{\ell}-1} \mathbbold{1}_{\SSS}(k) \geq \frac{1}{{\ell}} \sum_{k=0}^{{\ell}-1} \mathbbold{1}_{\SSS_m}(k) \geq 1 - \frac{1}{m},
\end{align}
yielding that $\SSS$ has density one. 

Notice that, by construction since $\SSS_{m+1} \subseteq \SSS_m$, we have that $\SSS \cap [i_m,\infty) \subseteq \SSS_m$. Therefore, for all $m>0$, we have
\begin{equation}
\lim_{\substack{i \to \infty \\ i \in \SSS}} a_i(U_m) = 0.
\end{equation}
We conclude the proof by noticing that any compact set $\bar{U}$ is contained in some $U_{\bar{m}}$, and we have $a_i(\bar{U}) \leq a_i(U_{\bar{m}})$ for all $i \in \N$.
\end{proof}

%% file: applications.tex
\section{Almost-Riemannian structures}\label{s:ARS}

We apply our results to a class of structures where the singularity admits a nice local description. This class is modelled on almost-Riemannian structures, introduced in \cite{ABS-GaussBonnet}. We provide here an explicit and local definition based on local coordinates. We refer the reader to \cite[Sec.\ 7]{QC} for a self-contained presentation close to our approach.

Let $N$ be a connected $n$-dimensional manifold, and let $S\subset N$ be an embedded hypersurface. We assume to be given a Riemannian metric $g$ on $N\setminus S$ such that, for all $p \in S$, there exist local coordinates $(x,z) \in \R\times \R^{n-1}$ and smooth vector fields
\begin{equation}
X_0 = \partial_x, \qquad X_i = \sum_{j=1}^{n-1} a_{ij}(x,z) \partial_{z_j},
\end{equation}
which are orthonormal for $g$ outside of $S$, and such that $\det a_{ij}(x,z) = 0$ if and only if $x=0$. This is a particular type of singular Riemannian metric on $N$, called \emph{almost-Riemannian structure} (ARS). Furthermore, we ask that there exist $m \in \N$ and smooth $\hat{a}_{ij}$ such that
\begin{equation}\label{eq:stronglyreg}
a_{ij}(x,z) = x^{m} \hat{a}_{ij}(x,z), \qquad \det \hat{a}_{ij}(0,z) \neq 0.
\end{equation}
In this case, on each local chart, we have
\begin{equation}
X_0 = \partial_x, \qquad X_i = x^m \hat{X}_i = x^{m} \sum_{j=1}^{n-1} \hat{a}_{ij}(x,z) \partial_{z_i},
\end{equation}
where $X_0,\hat{X}_1,\dots,\hat{X}_n$ have maximal rank also on the singular region. In particular we can introduce the regularized Riemannian  metric $\hat g$ in a neighborhood of $S$ as the metric with smooth orthonormal frame given by $\{X_0, \hat{X}_1,\ldots ,\hat X_n\}$. We denote with a hat all the quantities relative to this structure. In particular, the \emph{regularized measure} $\hat{\sigma}(S)$ of $S$ is defined as the surface measure of $S$ with respect to the regularized Riemannian measure.

\begin{definition}
A singular Riemannian structure on an $n$-dimensional manifold $N$ as above is called a strongly regular ARS of order $m$.
\end{definition}
Equivalently, a strongly regular ARS of order $m$ is an ARS whose Riemannian metric $g$ can be written, in a neighborhood of any point of $S$, as
\begin{equation}\label{eq:regmet}
g = dx^2 + x^{-2m} \hat{h}(x,z),
\end{equation}
where $\hat{h}(x,z)$ is a positive definite symmetric tensor, well defined also on the singularity.

\begin{rmk}
As a consequence of the theory developed in \cite{QC}, the Laplace-Beltrami operator of a strongly regular ARS  is essentially self-adjoint in $L^2(N\setminus S,d\mu_g)$. The same result holds more generally for \emph{regular structures}, introduced in \cite{QC}, that is, when the condition \eqref{eq:stronglyreg} is replaced by the weaker one
\begin{equation}\label{eq:kregularcondition}
\det a_{ij}(x,z) = x^k \phi(x,z), \qquad \phi(x,z) \neq 0.
\end{equation}
\end{rmk}

\subsection{Weyl's law for strongly regular ARS}
The restriction of $g$ to $\mathbb{M}= N\setminus S$ is a non-complete Riemannian manifold. The next proposition motivates the relevance of strongly regular ARS.
\begin{prop}\label{p:stronglyARS}
Any strongly regular ARS on a compact $n$-dimen\-sional manifold satisfies Assumption \ref{a:singularity}. Furthermore, as $\varepsilon \to 0$, we have
\begin{equation}
\vol(\M_\varepsilon^\infty) \sim 2 \hat{\sigma}(S) \times \begin{cases}
\varepsilon^{-(m(n-1)-1)} & m(n-1) > 1, \\
\log\varepsilon^{-1} & m(n-1) = 1,
\end{cases}
\end{equation}
where $m \in \N$ is the order of the strongly regular ARS.
\end{prop}
\begin{proof}
The non-complete Riemannian manifold $\M = N \setminus S$ has metric boundary given by one or two copies of $S$, depending whether the latter is one or two-sided. In the above local coordinates close to $S$ we have $\distmb(x,z) = |x|$. If $N$ is compact, $\distmb$ is smooth in a uniform neighborhood $U = \{\delta < \varepsilon_0\}$ of the metric boundary, i.e., (a) of Assumption \ref{a:singularity} is verified.

To compute curvature-type quantities, we adopt the following modified Einstein convention. Latin indices run from $1,\dots,n-1$, and repeated indices are summed on that range. The index $0$ is reserved for the variable $x$, i.e., $\partial_0 = \partial_x$. The non-vanishing structural functions are given by
\begin{equation}
[X_0, X_i] = c_{0i}^\ell X_\ell, \qquad [X_i,X_j] = c_{ij}^\ell X_\ell.
\end{equation}
Koszul's formula for the Levi-Civita connection in terms of orthonormal frames yields
\begin{gather}
\nabla_i X_j = \Gamma_{ij}^\ell X_\ell + \gamma_{ij} X_0, \qquad \nabla_0 X_i= \beta_{i\ell} X_\ell, \qquad \nabla_i X_0 = -\gamma_{i \ell} X_\ell, \\
\beta_{i\ell} = \frac{1}{2}(c_{0i}^\ell - c_{0\ell}^i), \qquad \gamma_{i\ell} = \frac{1}{2}(c_{0\ell}^i + c_{0i}^\ell), \qquad \Gamma_{ij}^\ell = \frac{1}{2}(c_{ij}^\ell + c_{\ell j}^i + c_{\ell i}^j).
\end{gather}
Notice that $\beta=-\beta^*$, $\gamma = \gamma^*$, while $\Gamma_{ij}^\ell = - \Gamma_{i\ell}^j$. From the definition
\begin{equation}
R(X,Y,Z,W) = g(\nabla_X \nabla_Y Z - \nabla_Y \nabla_X Z - \nabla_{[X,Y]} Z,W),
\end{equation}
we deduce the following formulas for the Riemann tensor
\begin{align}
R(X_i,X_j,X_k,X_\ell) & = \partial_i \Gamma_{jk}^\ell + \Gamma_{jk}^s \Gamma_{is}^\ell - \gamma_{jk}\gamma_{i\ell}-\partial_j \Gamma_{ik}^\ell - \Gamma_{ij}^s\Gamma_{js}^\ell +\gamma_{ik}\gamma_{j\ell} - c_{ij}^s \Gamma_{sk}^\ell, \label{eq:curvature1}\\
R(X_i,X_j,X_k,X_0) & = \Gamma_{jk}^\ell \gamma_{i \ell} + \partial_i \gamma_{jk} - \Gamma_{ik}^\ell \gamma_{j\ell} - \partial_j \gamma_{ik} - c_{ij}^s \gamma_{sk}, \label{eq:curvature2} \\
R(X_0,X_i,X_j,X_0) & =  \partial_0 \gamma_{i j} + \gamma_{i \ell} \beta_{\ell j} + \gamma_{j \ell} \beta_{\ell i} - \gamma_{i \ell}\gamma_{\ell j}, \label{eq:curvature3}
\end{align}
In particular $\mathrm{Sec}(X\wedge Y) = R(X,Y,Y,X)$ for any pair of unit orthogonal vectors $X,Y$. Furthermore, since $\distmb = |x|$, we have\begin{equation}\label{eq:hessian}
\Hess(\distmb)(X_i,X_j) =  -\mathrm{sgn}(x) \gamma_{ij}.
\end{equation}

In terms of the matrix $a$, the structural functions read
\begin{equation}\label{eq:c}
c_{0i}^\ell = a^{-1}_{\ell s} \partial_0 a_{s i}, \qquad c_{ij}^\ell = \left(a_{r i} \partial_{r} a_{s j} - a_{j r} \partial_r a_{s i}\right) a^{-1}_{s\ell}.
\end{equation}
Using \eqref{eq:c} one obtains
\begin{equation}
c_{0i}^\ell = \frac{m}{x} \mathbbold{1}_{i \ell} + \hat{c}_{0i}^\ell, \qquad c_{ij}^\ell = x^m \hat{c}_{ij}^\ell,
\end{equation}
where $\mathbbold{1}_{i \ell}$ is the Kronecker delta, which implies
\begin{equation}
\beta_{ij} = \hat{\beta}_{ij}, \qquad \gamma_{ij} = \frac{m}{x} \mathbbold{1}_{ij} + \hat{\gamma}_{ij}, \qquad \Gamma_{ij}^\ell = x^m \hat{\Gamma}_{ij}^\ell.
\end{equation}
From \eqref{eq:curvature1}-\eqref{eq:curvature3} we obtain
\begin{align}
R(X_i,X_j,X_k,X_\ell) & = -\frac{m^2}{x^2}(\mathbbold{1}_{jk}\mathbbold{1}_{i\ell}-\mathbbold{1}_{ik}\mathbbold{1}_{j\ell}) + O\left(\frac{1}{|x|}\right), \\
R(X_i,X_j,X_k,X_0) & = O(1), \\
R(X_0,X_i,X_j,X_0) & = - \frac{m(m+1)}{x^2} \mathbbold{1}_{ij} + O\left(\frac{1}{|x|}\right),
\end{align}
and
\begin{equation}\label{eq:koszul-hess}
\Hess(\distmb)(X_i,X_j) = -\frac{m}{|x|}\mathbbold{1}_{ij} + O(1).
\end{equation}
In particular points (b) and (c) of Assumption~\ref{a:singularity} are verified, and the convexity condition is strict. Moreover, point (d) is valid thanks to Proposition \ref{prop:inj-strict-convex}. Finally, the volume asymptotics follows from a straightforward computation.
%To compute the volume asymptotics, assume without loss of generality that $S$ is contained in a single chart. Then we have, $\vol(\M_\varepsilon^\infty) \sim \vol(\M_\varepsilon^{\varepsilon_0})$, and
%\begin{align}
%\vol(\M_\varepsilon^{\varepsilon_0}) & = \int_{\varepsilon \leq |x|\leq \varepsilon_0} dx \int_{S} \sqrt{\det(a a^*)}(x,z) dz \\
%& = \int_{\varepsilon \leq |x|\leq \varepsilon_0} \frac{1}{x^{m(n-1)}} dx \int_{S} \sqrt{\det(\hat{a}\hat{a}^*(x,z))}dz \\
%& = \int_{\varepsilon \leq |x|\leq \varepsilon_0} \frac{1}{x^{m(n-1)}}dx  \left(\int_{S} \sqrt{\det(\hat{a}\hat{a}^*(0,z))}dz\right)(1+O(\varepsilon)) \\
%& =  2  (1+O(\varepsilon))\hat{\sigma}(S) \int_\varepsilon^{\varepsilon_0} \frac{1}{x^{m(n-1)}}dx,
%\end{align}
%from which the result follows. %The factor $2$ is due to the fact that the coordinate $x$ can assume positive and negative values.
\end{proof}

We can then apply our theory to strongly regular ARS. We recall that the notation $f(\lambda)\asymp g(\lambda)$ means that $f(\lambda)/g(\lambda)$ has finite and positive $\limsup$ and $\liminf$, as $\lambda\to\infty$.	

\begin{theorem}[Weyl's law for strongly regular ARS]\label{t:ARS}
Consider the Laplace-Beltrami operator of a strongly regular ARS of order $m$ on an $n$-dimensional compact manifold. Then, if $(n,m)\neq (2,1)$, we have
\begin{equation}
  N(\lambda)\asymp \lambda^{(n-1)(m+1)/2}, \qquad  \lambda \to \infty.
\end{equation}
On the other hand, if $(n,m)= (2,1)$, we have
    \begin{equation}
 N(\lambda) \sim  \frac{|\hat\sigma(S)|}{4\pi} \lambda \log\lambda.
    \end{equation}
\end{theorem}
\begin{proof}
By Proposition \ref{p:stronglyARS}, Assumption \ref{a:singularity} is verified, and it can be easily seen that the volume function $\vv(\lambda) = \vol(\M_{1/\sqrt{\lambda}})$ satisfies
\begin{equation}
\vv(\lambda) \sim 2\hat{\sigma}(S) \times \begin{cases}
\lambda^{(m(n-1)-1)/2} & m(n-1)>1, \\
\log\lambda & m(n-1)=1.
\end{cases}
\end{equation}
In the first case, the result follows from Theorem \ref{t:weyl} while, in the second case, $\vv(\lambda)$ is slowly varying and we can apply Theorem \ref{t:weylexact}.
\end{proof}

\subsection{Examples}
We conclude this section with two examples. The first one shows that a general non-strongly regular ARS does not satisfy Assumption~\ref{a:singularity}. In particular, on ARSs all geometric quantities can have an arbitrarily fast polynomial explosion to $\pm \infty$. The second example is an ARS structure that satisfies Assumption~\ref{a:singularity} but that is not regular.

%The next example shows that the results of Proposition \ref{p:stronglyARS} do not hold without the strongly regular condition \eqref{eq:stronglyreg}. In particular, one can have ARS where all geometric quantities have an arbitrarily fast polynomial explosion to $\pm \infty$.

\begin{example}[Worst case curvature explosion]\label{e:worst}
Let $k\geq 1$, and consider the structure defined by declaring the following vector fields to be orthonormal
\begin{equation}
X_0 = \partial_x, \qquad X_1 =  \partial_{z_1}+x\partial_{z_2}, \qquad X_2 = x^{k}\partial_{z_2},
\end{equation}
where $(x,z) \in \R \times \R^{2}$. In other words, $X_i = \sum_j a_{ij}\partial_{z_j}$, for $i=1,2$, with
\begin{equation}
a = \begin{pmatrix} 1 & 0 \\x & x^k
\end{pmatrix}.
\end{equation}
This structure is not strongly regular. %Let us show that it does not satisfy Assumption~\ref{a:singularity}.
We  use the formalism introduced in the proof of Proposition~\ref{p:stronglyARS}.
In particular, letting $C_{i\ell} = c_{0i}^\ell$, we have
\begin{equation}
C = (a^{-1}\partial_0 a)^* = \begin{pmatrix}
0 & \tfrac{1}{x^k} \\0 & \tfrac{k}{x}
\end{pmatrix}.
\end{equation}
It follows that
\begin{equation}
\beta  = \frac{1}{2}(C-C^*) = \begin{pmatrix}
0 & \tfrac{1}{2x^k} \\ -\tfrac{1}{2x^k} & 0
\end{pmatrix} 
,\quad\text{and}\quad
\gamma = \frac{1}{2}(C+C^*)= \begin{pmatrix}
0 & \tfrac{1}{2x^k} \\ \tfrac{1}{2x^k} & \tfrac{k}{x}
\end{pmatrix}.
\end{equation}
Rewriting \eqref{eq:curvature3} in this notation, we obtain, for $i=1,2$,
\begin{equation}
\mathrm{Sec}(X_0\wedge X_i)  = (\partial_0\gamma + 2\gamma\beta- \gamma^2)_{ii}  = \mathrm{diag}\left(-\frac{3}{4x^{2k}},\frac{1}{4x^{2k}}-\frac{k(k+1)}{x^2}\right).
\end{equation}
Thus, this structure does not satisfy the curvature assumptions of Assumption~\ref{a:singularity} as soon as $k\ge 2$. Furthermore, even the convexity assumption is not satisfied: by \eqref{eq:hessian}, the eigenvalues of $\Hess(\distmb)$ are $h_\pm = \pm \tfrac{1}{2|x|^k}(1+o(1))$ as $x\to 0$.
\end{example}

\begin{example}
Let $N$ be a compact $2$-dimensional manifold, and  $S \simeq \mathbb{S}^1$ be a smooth embedded sub-manifold of $N$. Equip $N$ with an ARS whose local orthonormal frame, in a neighborhood $U =(-1,1)\times \mathbb S^1$ of $S$, reads
\begin{equation}
  X_0 = \partial_x,\qquad
  X_1 = x\left(x^{2}+\sin^2(\theta/2)\right) \partial_\theta.
\end{equation}
This structure is not regular, since \eqref{eq:stronglyreg} is not satisfied uniformly for $\theta\in\mathbb S^1$. Nevertheless, by \eqref{eq:curvature3} and \eqref{eq:hessian} we have
 \begin{equation}    
 \Sec(X_0 \wedge X_1) = -\frac{2}{x^2} - \frac{10x^2+2\sin^2({\theta }/{2})}{(x^2+\sin^2({\theta }/{2}))^2}
 \quad  \text{and}\quad
    \Hess(\distmb) =-\frac{1}{|x|}-\frac{2|x|}{x^2+\sin^2({\theta }/{2})}.
  \end{equation}
Then, on $U$ it holds
  \begin{equation}
    -\frac{13}{\distmb^2} \le \Sec \le 0, \qquad\text{and}\qquad \Hess(\distmb)< 0. 
  \end{equation}
In particular, by Proposition \ref{prop:inj-strict-convex}, this singular Riemannian structure satisfies Assumption~\ref{a:singularity} and thus we can apply Theorem~\ref{t:weyl}.
To this aim, we only need to estimate the volume function $\lambda\mapsto\vv(\lambda)$, which can be readily done by integrating the measure $d\mu=\frac{dx d\theta}{|x|(x^2+\sin^2(\theta/2))}$. Indeed, we have
  \begin{equation}
      \int_{[\varepsilon,1]\times\mathbb S^1} d\mu 
        = \int_{\varepsilon}^{1}\frac{2 \pi }{x^2 \sqrt{x^2+1}}\,dx 
        =\frac{2\pi}\varepsilon + O(1),
  \end{equation}
  whence $\vv(\lambda) \sim 4 \pi \lambda^{1/2}$. This immediately yields $   N(\lambda)\asymp \lambda^{3/2}$ as $\lambda \to \infty$.
\end{example}

%% file: geometricbounds.tex
\section{Auxiliary geometric estimates}\label{a:auxiliary}

On the simply connected $n$-dimensional Riemannian space form $M_K$ of curvature $K\in\R$, the heat kernel depends only on $t$ and on the distance $r=d(q,p)$, and thus, with a slight abuse of notation, we denote it by $E_K(t,r)$.
Here and below, $r \in [0,\pi/ \sqrt{K}]$, with the convention that $\pi / \sqrt{K} = +\infty$ if $K \leq 0$.

\begin{lemma}\label{l:pKreste}
For all $T>0$ there exists a constant $C>0$, depending only on $n$ and $T$, such that
\begin{equation}
|(4 \pi t)^{n/2} E_K(t,0) - 1 | \leq C |K| t, \qquad \forall t \leq T/|K|.
\end{equation}
\end{lemma}

\begin{proof}
If $K=0$, the estimate is trivially verified. Let us consider $K \neq 0$. For a Riemannian metric $g$ and $\alpha>0$, let $g_{\alpha} := \alpha^2 g$. Then, $\Sec(g_\alpha) = \alpha^{-2}\Sec(g)$, and $E_{g_\alpha}(t,p,q) = {\alpha^{-n}} E_{g}(t/\alpha^2,p,q)$. This immediately implies
\begin{equation}\label{eq:rescaling}
(4\pi t)^{n/2} E_K(t,r) = (4\pi t |K|)^{n/2}E_{\pm 1}(t |K|,r \sqrt{|K|}),
\end{equation}
where $\pm 1$ is the sign of $K$. Moreover, by the Minakshisundaram-Pleijel asymptotics,\footnote{We refer to the following simplified statement, valid for any complete $n$-dimensional Riemannian manifold: for all $T>0$ and $q \in M$ there exists $C>0$ such that $\left|(4 \pi t)^{n/2}E(t,q,q) - 1\right| \leq C t$, for all $t \in (0,T]$. For a proof in the compact case, see e.g.\ \cite[Prop.\ 3.23]{Rosenberg}. The extension to the non compact case is done via a localization argument exploiting \eqref{eq:Grig} and Varadhan's formula.} we deduce that, for all $T>0$, there exist a constant $C>0$ such that
\begin{equation}
|(4\pi t )^{n/2} E_{\pm 1}(t,0) - 1| \leq C t, \qquad \forall t \leq T,
\end{equation}
where $C$ depends only on $n$ and $T$. This and \eqref{eq:rescaling} prove the statement.
\end{proof}

\begin{lemma}\label{l:deduction}
Let $K \geq 0$, and let $B_K(r)$ be the ball of radius $r \leq \pi/\sqrt{K}$ on the simply connected space form with constant curvature equal to $K$ and dimension $n$. Then, there exists a constant $C>0$, depending only on the dimension, such that
\begin{equation}
\vol(B_K(r)) \geq C r^n, \qquad \forall r \le \pi/\sqrt{K}.
\end{equation}
\end{lemma}
\begin{proof}
Since $K \geq 0$, by the Bishop-Gromov inequality, $r \mapsto \vol(B_K(r)) / \vol(B_0(r))$ is non-increasing. Hence, the rescaling argument used in the proof of Lemma~\ref{l:pKreste} yields
\begin{equation}
\frac{\vol(B_K(r))}{\vol(B_0(r))} = \frac{\vol(B_1(r \sqrt{K}))}{\vol(B_0(r \sqrt{K}))} \geq \frac{\vol(B_1(\pi))}{\vol(B_0(\pi))}.
\end{equation}
To conclude the proof it suffices to observe that $\vol(B_0(r)) = r^n\vol(B_0(1))$. In particular, this yields $C = \vol(B_1(\pi))/\pi^n$.
\end{proof}

In the next lemma, we show that, for any ball there always exists a spherical sector which points away from the boundary and whose size does not depend on the point. This yields a uniform lower bound to the measure of sufficiently small balls.
\begin{lemma}\label{l:cono}
Let $(M,g)$ be a {complete} $n$-dimensional Riemannian manifold with boundary. Let $H\geq 0$ such that $-H \leq \mathrm{Hess}(\distb)$ for $\distb < \inj_{\partial}(M){/2}$ and set
\begin{equation}
r_0= \min\left\{\inj(M),\frac{\inj_{\partial}(M)}{4},\frac{1}{H}\right\}.
\end{equation}
Then, for any $o \in M$ and $r \leq r_0$, there exists an open set $S_o(r) \subset B_o(r)$ such that 
\begin{itemize}
\item if $B_o(r)$ does not intersect $\partial M$, then $S_o(r) = B_o(r)$;
\item if $B_o(r)$ intersects $\partial M$, then the closest point of $S_o(r)$ to $\partial M$ is $o$.
\end{itemize}

Let, moreover, $K\ge 0$ be such that $\Sec(g)\le K$ on $S_o(r)$. Then, there exists a constant $C\in(0,1/2)$, depending only on $n$, such that
\begin{equation}\label{eq:bkr}
\vol(B_{o}(r)) \geq \vol(S_o(r)) \geq C \vol(B_K(r)), \qquad  \forall r \leq r_0.
\end{equation}
\end{lemma}
\begin{proof}
Fix $r \leq r_0$. If $\distb(o) > r$, the ball does not intersect the boundary, and we set $S_o(r) = B_o(r)$. By the curvature upper bound, and since the balls lie within the injectivity radius from their center, we have that their volume is bounded from below by the volume of the ball with the same radius in the simply connected space form with constant curvature equal to $K$, which yields \eqref{eq:bkr} with $C=1$.

On the other hand, if $\distb(o) \leq r$, the ball hits $\partial M$. The condition $r \le \inj_{\partial}(M)/4$ implies that $B_o(r)$ lies in the region where $\distb$ is smooth and $- H \leq \mathrm{Hess}(\distb)$. Consider a length parametrized geodesic $\gamma$ emanating from $o$, in a direction so that $\cos\theta := g(\dot\gamma,\nabla \distb)>0$. It then holds $\theta\in(-\pi/2,\pi/2)$, and
\begin{equation}
\distb(\gamma(t)) \geq - \frac{H}{2}\sin^2(\theta) t^2 + \cos(\theta) t + \distb(o), \qquad \forall t \leq r.
\end{equation}
Therefore, minimizing geodesics emanating from $o$ and with length smaller than $r$ do not cross $\partial M$ provided that, e.g., 
\begin{equation}
\cos\theta \geq \frac{Hr}{2}.
\end{equation}
Thanks to the assumption $r \leq 1/H$ the above inequality holds if $|\theta| < \pi/3$. Let $S_o(r) \subset B_o(r)$ be the corresponding spherical sector of radius $r$. 
By construction, $o$ is its closest point to $\partial M$. 
Since $r \leq r_0 \leq \inj(M)$, we can fix normal polar coordinates $(s,\Omega) \in [0,r_0]\times \mathbb{S}^{n-1}$ at $o$. Therefore,
\begin{equation}
\vol(B_o(r)) \geq \vol(S_o(r)) = \int_{\bar{S}(r)} s^{n-1} A(s,\Omega) ds d\Omega,
\end{equation}
where $\bar{S}(r)$ is the Euclidean spherical sector corresponding to $S_o(r)$ in these coordinates, and $s^{n-1}A(s,\Omega)$ is the Jacobian determinant of the exponential map with base $o$. By standard comparison arguments, the assumption $\Sec(g) \leq K$ yields $A(s,\Omega) \geq A_K(s)$, where the latter is the corresponding object on the $n$-dimensional space form with constant curvature equal to $K$. Hence,
\begin{equation}
\vol(S_o(r)) \geq \int_{\bar{S}(r)} s^{n-1} A_K(s) d s d\Omega.
\end{equation}
Without loss of generality, we can fix coordinates $(\theta,\varphi) \in (-\pi/2,\pi/2) \times \mathbb{S}^{n-2}$ such that $\bar{S}(r) = \{ |\theta| < \pi/3,\; s < r\}$. In these coordinates $d\Omega = \sin(\theta)^{n-2} d\theta d\varphi$, where $d\varphi$ is the standard measure on $\mathbb{S}^{n-2}$. Therefore,
\begin{equation}
\vol(S_o(r)) \geq \int_0^r s^{n-1} A_K(s)  ds \int_0^{\pi/3} \sin(\theta)^{n-2}\vol(\mathbb{S}^{n-2}) d\theta  = C \vol(B_K(r)).
\end{equation}
A simple symmetry argument implies that $C\in(0,1/2)$.
\end{proof}

\begin{lemma}[Li-Yau inequality]\label{l:li-yau}
  Let $(M, g)$ be a complete $n$-dimensional Riemannian manifold with boundary (convex, in the Neumann case), and $\Ric(M)\ge -K(n-1)$, for some $K \geq 0$. Then there exist $C_1,C_2,C_3>0$, depending only on $n$, such that
  \begin{equation}\label{eq:LYlemma}
    E^{\pm}(t,p,q)\le \frac{C_1}{\sqrt{\vol(B_p(\sqrt{t}))\vol(B_q(\sqrt{t}))}}e^{C_2 Kt - C_3\frac{d^2(p,q)}{4t}},\quad \forall(t,p,q)\in \R_+\times M \times M.
  \end{equation}
Furthermore, assume that $\Sec(M)\leq K$, and let
\begin{equation}
\sqrt{t_0} = \min \left\{ \inj(M),\frac{\inj_{\partial}(M)}{4}, \frac{\pi}{\sqrt{K}}\right\}.
\end{equation}
Then, there exists a constant $C_4>0$, depending only on $n$, such that
\begin{equation}\label{eq:LY-diag}
    (4 \pi t)^{n/2} E^{\pm}(t,p,q)\le C_4 e^{- C_3\frac{d^2(p,q)}{4t}}, \qquad \forall (t,p,q) \in (0,t_0) \times M \times M.
\end{equation}
\end{lemma}
\begin{proof}
The first inequality is the celebrated estimate \cite[Thm.\ 3.2]{Li1986} by Li and Yau, where the parameters $\varepsilon$ and $\alpha$ are fixed in the allowed ranges.
  
To prove \eqref{eq:LY-diag}, we uniformly bound from below the volumes appearing in the denominator of \eqref{eq:LYlemma}. By Lemma \ref{l:stdcomparisonforH}, on the region $\{\distb<\inj_\partial(M)/2\}$ it holds $\Hess(\distb)\geq -H$ for some $H>0$ satisfying $\tfrac{\sqrt{t_0}}{10}\leq \tfrac{1}{H}$. Thus, we have $\tfrac{\sqrt{t}}{10}\leq \min \left\{ \inj(M),\frac{\inj_{\partial}(M)}{4},  \frac{1}{H},\frac{\pi}{\sqrt{K}}\right\}$. Then, we can apply both Lemma~\ref{l:cono} and \ref{l:deduction}, and $\vol(B_K(\sqrt{t}/10)) \geq C t^{n/2}$ for a constant $C>0$ depending only on $n$.
\end{proof}

\begin{lemma}\label{l:stdcomparisonforH}
Let $M$ be a complete manifold with boundary and let $K\geq 0$ such that $\Sec(M) \geq -K$. Then, on the region $\{\distb< \tfrac{\inj_{\partial}(M)}{2}\}$, it holds
\begin{equation}
  \label{eq:H-plus}
-H_-:=-\sqrt{K}\coth(r_0 \sqrt{K}) \leq \Hess(d_\partial), \qquad r_0 := \min \left\{\inj(M),\tfrac{\inj_\partial(M)}{4}\right\}.
\end{equation}
Furthermore, if $\partial M$ is convex, it also holds
\begin{equation}
  \label{eq:H-minus}
\Hess(d_\partial) \leq \sqrt{K} \tanh( \sqrt{K}\inj_\partial(M)/2)=: H_+.
\end{equation}
In particular, setting $H = \max\{H_+,H_-\}$, it holds $|\Hess(\distb)|\leq H$ and
\begin{equation}
 \frac{1}{10} \min \left\{\inj(M),\frac{\inj_\partial(M)}{4},\frac{\pi}{\sqrt{K}}\right\} \leq \frac{1}{H} \leq \min \left\{\inj(M),\frac{\inj_\partial(M)}{4},\frac{\pi}{\sqrt{K}}\right\}.
\end{equation}
\end{lemma}
\begin{proof}
  The last part of the statement follows by elementary computations, using \eqref{eq:H-plus} and \eqref{eq:H-minus}. In turn these bounds are obtained via Riccati comparison, as we now detail.

To prove \eqref{eq:H-plus}, observe that $r_{0}$ is such that for any $p\in\partial M$ we can find a point  $q\in M$ such that the ball $B_q(r_0)$ is contained in $M$ and is tangent to $\partial M$ at $p$. Letting $r = d(q,\cdot)$, this implies that $\Hess(\distb)\geq -\Hess(r)$ at $p$. By the lower curvature bound, $\Hess(r)$ is not greater than the same quantity on a model space of constant curvature $-K$, which is given by $\sqrt{K}\coth(\sqrt{K} r_0)$ (i.e., the value at $r_0$ of the solution of $h'(t) +h(t)^2-K=0$ with initial condition $h(t) \sim \tfrac{1}{t}$ as $t\to 0$).  This proves the required bound on $\partial M$. To obtain the bound on the region $\{\distb < \tfrac{\inj_{\partial}(M)}2\}$ it suffices to repeat the same construction by replacing $\partial M$ with the level set $\distb = c$ for all $c < \inj_\partial(M)/2$.

To obtain the bound \eqref{eq:H-minus} it suffices to observe that, thanks to $\Hess(\distb)\leq 0$ on $\partial M$ and  $\Sec(M) \geq -K$, Riccati comparison implies that $\Hess(\distb) \leq \sqrt{K}\tanh(\sqrt{K} c)$ on the level set $\distb = c<\inj_\partial(M)$ (here, the right hand side is the value at $c$ of the solution of $h'(t) +h(t)^2-K=0$ with $h(0)=0$).
\end{proof}

The following theorem was suggested in \cite[p.\ 69]{AlexanderBishop} for complete Riemannian structures with curvature bounded above and injectivity radius bounded below.
\begin{lemma}\label{l:injsubmanifold}
 Let $(\M,g)$ be a Riemannian manifold, possibly non-complete. Let $(Z,h)$ be a closed submanifold with bounded second fundamental form $|II| \leq H$. Assume that, on a tube $T_D$ around $Z$ of radius $D>0$ at positive distance from the metric boundary of $\M$, we have $\Sec(T_D,g) \leq K$ and $\inj(T_D,g) \geq I$. Then it holds
\begin{equation}
\inj(Z,h) \geq \min\left\{ \frac{\pi}{\sqrt{K +H^2}},\frac{\pi}{2\sqrt{K}}, I, D \right\},
\end{equation}
with the convention that $\pi/\sqrt{K} = \infty$ if $K\leq 0$.
\end{lemma}
\begin{proof}
By Gauss' equation, for all $X,Y,U,V \in TZ$, we have
\begin{equation}
R^g(X,Y,U,V) = R^h(X,Y,U,V) + II(X,U) II(Y,V)- II(Y,U) II(X,V).
\end{equation}
It follows that $\Sec(Z,h) \leq K+H^2$. By Klingenberg's Lemma,
\begin{equation}\label{eq:kkk}
\inj(Z,h) \geq \min\left\{\frac{\pi}{\sqrt{K+H^2}},\frac{\ell(\gamma)}{2}\right\},
\end{equation}
where $\ell(\gamma)$ is length the shortest non-trivial closed geodesic in $(Z,h)$. Let $\gamma$ be such a geodesic, parametrized with unit speed. Its curvature in $(\M,g)$ is bounded by
\begin{equation}
|\nabla_{\dot\gamma}^g \dot\gamma| = | II(\dot\gamma,\dot\gamma)| \leq H.
\end{equation}
Assume that $\ell(\gamma) < \min \{2 I,\pi/\sqrt{K},2D\}$. In this case, $\gamma$ lies in a ball $B_R$ of $(\M,g)$ with radius $R \leq \tfrac{\pi}{2\sqrt{K}}$ within the injectivity radius of its center, not touching the metric boundary of $(\M,g)$, and the curvature on $B_R$ is bounded above by $K$. By \cite[p.\ 100, Proposition 6.4.6 (ii)]{Buser-Karcher-almostflat} for any two points in $B_R$ there exists a unique geodesic joining them, of length $\leq 2R$, and all contained in the interior of $B_R$, In other words $B_R$, considered as a metric space with the induced length metric, has the unique geodesic property (any two points are joined by a unique geodesic). As a corollary of Rauch's Theorem \cite[1.30]{CheegerEbin}, geodesic triangles in $B_R$ are thinner than corresponding ones on the model space with constant curvature $K$. Thus, $B_R$ is a $\mathrm{CAT}(K)$ space.\footnote{We recall that a $\mathrm{CAT}(K)$ space (also called an $R_K$ domain) is a metric space in which any two points are joined by a unique
geodesic and, for any geodesic triangle of perimeter less than $2\pi/\sqrt{K}$, the distance between points on the triangle is at most equal to the distance between the corresponding points on
the model triangle in the simply connected $2$-dimensional space of curvature $K$.}

The length of closed curves with geodesic curvature bounded from above by $H$ in a $\mathrm{CAT}(K)$ space can be bounded from below in terms of the length of the corresponding circles with constant curvature $H$ on the $2$-dimensional simply connected space form with constant curvature $K$, see \cite[Cor.\ 1.2(c)]{AlexanderBishop}. We thus obtain that
\begin{equation}
\frac{\ell(\gamma)}{2} \geq \frac{\pi}{\sqrt{K+H^2}}.
\end{equation} 
We conclude easily, see also \cite[Thm.\ 1.3]{Alex-Bishop-Injectvity}.
\end{proof}

\section{Compactness of the resolvent}\label{app:compactness}

In the following, $\delta : \M \to (0,\infty)$ is a general smooth function, even though we only need the case in which $\delta$ is a distance from the metric boundary.

\begin{prop}[Hardy inequality]\label{p:hardy-appendix}
Let $\M$ be a non-complete Riemannian manifold. Let $\delta : \M \to (0,\infty)$ be smooth. Let $a>0$ and $U := \{x\in \M\mid \delta(x)\leq a\}$. Assume that $\Delta\delta \leq 0$ on $U$ and that $|\nabla\delta|$ is bounded on $U$. Then it holds:
  \begin{equation}\label{eq:hardy-appendix}
    \int_{U} |\nabla u|^2\,d\mu_g \geq \frac{1}{4}\int_{U} \frac{|u|^2}{\delta^2} |\nabla\delta|^2\,d\mu_g, \qquad \forall u \in H^1_0(\M),
  \end{equation}
  where $H^1_0(\M)$ is the closure in the $W^{1,2}$ norm of the space of functions $\{\phi \mid \phi \in C^\infty_c(\M)\}$. 
\end{prop}
\begin{proof}
We observe that $U$ a measurable subset. Since $|\nabla\delta|$ is bounded, it is sufficient to prove \eqref{eq:hardy-appendix} for (the restriction to $U$) of $u\in C^\infty_c(\M)$. Set $h:=u \delta^{-1/2}$. We find
\begin{equation}
|\nabla u|^2 \geq \frac{u^2}{4\delta^2}|\nabla\delta|^2 + \frac{1}{2}g(\nabla \delta ,\nabla h^2).
\end{equation}
Assume first that $a$ is a regular value for $\delta$, in which case $\M^a_0$ is a smooth manifold with boundary $\partial U = \{\delta = a\}$, with induced measure $d\sigma_g$. %Notice that all those boundaries are mean convex that is the second fundamental form of $\partial \M^a$ is non-negative when the above orientation is chosen. 
Integrating the above inequality, and using the divergence theorem, we obtain
\begin{equation}
\int_{U} |\nabla u|^2\, d\mu_g \geq \frac{1}{4}\int_{U} \frac{u^2}{\delta^2}|\nabla\delta|^2\,d\mu_g - \frac{1}{2} \int_{U} \frac{u^2}{\delta}\Delta\delta\,d\mu_g + \frac{1}{2}\int_{\partial U} \frac{u^2}{\delta} |\nabla\delta| \;d\sigma_g  \geq  \frac{1}{4}\int_{U} \frac{u^2}{\delta^2}|\nabla\delta|^2\,d\mu_g.
\end{equation}
This proves \eqref{eq:hardy-appendix} when $a$ is a non-critical value for $\delta$. If $a$ is a critical value, then by Sard theorem we can find $(a_n)_{n\in \N}$ with $a_n\uparrow a$ of non-critical values for which the inequality holds, and we conclude by the monotone convergence theorem.
\end{proof}

We use the previous theorem only in the case in which $\delta$ is the distance from the metric boundary, for which $|\nabla \delta|=1$. A non-complete Riemannian manifold $(\M,g)$ has \emph{regular metric boundary} if the distance $\distmb$ from the metric boundary is smooth in a neighborhood $U= \{\delta < \varepsilon_0\}$ of the metric boundary (this is point (a) in Assumption \ref{a:singularity}).

The next statement is a simplified version of \cite[Prop.\ 3.7]{QC}.

\begin{theorem}[Compact embedding]\label{t:compact-emb}
  Let $\M$ be a non-complete Riemannian manifold with compact metric completion and regular metric boundary. Assume, moreover, that the boundaries $\partial \M_\varepsilon$ are mean convex for sufficiently small $\varepsilon$. 
Then $H^1_0(\Omega)$ and $H^1(\Omega)$ compactly embed in $L^2(\Omega)$, where $\Omega = \M_a^b$  for  $0\leq a< b \leq \infty$.
\end{theorem}

\begin{proof}
Although the only non-standard case is the case $a=0$, we provide a unified proof for all $a$.
  Let $(u_n)_n\subset H^1_0(\Omega)$ be such that $\|u_n\|_{H^1(\Omega)}\le C$ for some $C>0$. In order to find a subsequence converging in $L^2(\Omega)$, we consider separately the behavior close and far away from the metric boundary. For a fixed $\varepsilon>0$ sufficiently small, consider two Lipschitz functions $\phi_1,\phi_2:\M\to [0,1]$ such that $\phi_1+\phi_2\equiv 1$, $\phi_1\equiv 1$ on $\M_0^{\varepsilon/2}$, $\supp\phi_1\subset \M_0^\varepsilon$, and $|\nabla\phi_i|\le M$ for some $M>0$. Define $u_{n,i}=\phi_i u_n$, so that, with a slight abuse of notation, $u_{n,1}\in H^1_0(\M_0^\varepsilon \cap \Omega)$ and $u_{n,2}\in H^1_0(\M_{\varepsilon/2}^\infty\cap\Omega)$. 
  
  By a density argument, Leibniz rule, and Young inequality, for $i=1,2$ we have
  \begin{equation}
    \int_{\Omega} |\nabla u_{n,i}|^2\,d\mu_g 
    \le 2 \int_{\Omega}\left(|\nabla\phi_i|^2 |u|^2 + \phi_i^2|\nabla u|^2 \right)\,d\mu_g.
  \end{equation}
  By the fact that $\|u_n\|_{H^1(\Omega)}\le C$ and that $\phi_i$ is uniformly Lipschitz, the above implies that, up to enlarging $C>0$, it holds $\|u_{n,1}\|_{H^1(\M_0^\varepsilon\cap\Omega)}\le C$ and $\|u_{n,2}\|_{H^1(\M_{\varepsilon/2}^\infty\cap\Omega)}\le C$.
 Since $\M_{\varepsilon/2}^\infty \cap \Omega$ is relatively compact in $(\M,g)$, by \cite[Cor.\ 10.21]{Grigoryan-book} we have that $H^1_0(\M_{\varepsilon/2}^\infty\cap\Omega)$ compactly embeds in $L^2(\M_{\varepsilon/2}^\infty\cap\Omega)$. 
  Thus, $(u_{n,2})_n$, being bounded in $H^1_0(\M_{\varepsilon/2}^\infty\cap\Omega)$, admits a convergent subsequence in $L^2(\Omega)$.
  
  On the other hand, by the Hardy inequality of Proposition~\ref{p:hardy-appendix}, we have
  \begin{equation}
    \|u_{n,1}\|_{L^2(\M_0^\varepsilon\cap \Omega)}^2 = \int_{\M_0^\varepsilon \cap \Omega}|u_{n,1}|^2\,d\mu_g 
    \le 4\varepsilon^2\int_{\M_0^{\varepsilon} \cap \Omega}|\nabla u_{n,1}|^2 \,d\mu_g \le 4C\varepsilon^2,
  \end{equation}
  where we used the boundedness of $(u_{n,1})_n$ in $H^1(\M_0^\varepsilon\cap \Omega)$.
  Then, by choosing $\varepsilon = \varepsilon_k=(2Ck)^{-1}$, we obtain that for any $k\in \N$ there exists a subsequence $n\mapsto \gamma_k(n)$ such that $u_{\gamma_k(n)} = u_{\gamma_k(n),1}+u_{\gamma_k(n),2}$ with $\|u_{\gamma_k(n),1}\|\le 1/k$ and $(u_{\gamma_k(n),2})_n$ convergent in $L^2(\Omega)$. A diagonal argument yields the existence of a subsequence of $(u_n)_n$ convergent in $L^2(\Omega)$, proving the compact embedding of $H^1_0(\Omega)$ in $L^2(\Omega)$ (see e.g.\ \cite[Prop.\ 3.7]{QC}).
  
To prove the analogous statement for $H^1(\Omega)$, we follow the same steps but, in this case, $u_{n,2}\in H^1(\M_{\varepsilon/2}^\infty\cap\Omega)$. 
Since $\M_{\varepsilon/2}^\infty\cap\Omega$ can be seen as a relatively compact subset with smooth boundary of a larger complete Riemannian manifold, $H^1(\M_{\varepsilon/2}^\infty\cap\Omega)$ compactly embeds in $L^2(\M_{\varepsilon/2}^\infty\cap\Omega)$.\footnote{It follows by the arguments of \cite[Cor.\ 10.21, 2nd proof]{Grigoryan-book} and Rellich-Kondrachov \cite[Thm.\ 6.3]{Adams2003}.}
\end{proof}

\begin{corollary}\label{cor:discrete-spec}
  Let $\M$ be a non-complete Riemannian manifold with compact metric completion and regular metric boundary. Assume, moreover, that the boundaries $\partial \M_\varepsilon^\infty$ are mean convex for sufficiently small $\varepsilon$. Then the resolvents $(\Delta^{\pm}_{\Omega}-z)^{-1}$ of the Dirichlet or Neumann Laplace-Beltrami operators are compact for any $z>0$, where $\Omega = \M_a^b$  for  $0\leq a< b \leq \infty$. In particular, the spectra of $\Delta^{\pm}_{\Omega}$ are discrete.
\end{corollary}

\begin{proof}
  By Theorem~\ref{t:compact-emb}, $H^1(\Omega)$ is compactly embedded in $L^2(\Omega)$. Since the domain $\dom(\Delta^\pm_\Omega)$ is contained in $H^1(\Omega$), this implies the compactness of the resolvent. To this effect, and for completeness sake, we replicate  the argument of \cite[Thm.\ 10.20]{Grigoryan-book}. 
  
  Since $\Delta^{\pm}_{\Omega}$ is a non-positive operator, its resolvent set contains $(0,+\infty)$. Thus, for $z>0$, $R_z:=(\Delta^{\pm}_{\Omega}-z)^{-1}$ is a bounded self-adjoint operator in $L^2(\Omega)$. Moreover, for any $\psi\in L^2(\Omega)$ we have $u:=R_z\psi\in \dom(\Delta^{\pm}_{\Omega})\subset H^1({\Omega})$, whence
  \begin{equation}
    \int_{{\Omega}}|\nabla u|^2\,d\mu_g +z \int_{{\Omega}}|u|^2\,d\mu_g
    = -\int_{{\Omega}}\bar{u}\left(\Delta^{\pm}_{\Omega}u -zu\right) \,d\mu_g 
    = -\int_{\Omega} \bar{u} \psi \,d\mu_g.
  \end{equation}
  By the Cauchy-Schwarz inequality this implies
  \begin{equation}
    \min\{1,z\}\|u\|^2_{H^1} \le \|u\|_{L^2}\|\psi\|_{L^2}\le \|u\|_{H^1}\|\psi\|_{L^2}.
  \end{equation}
  We then get $\|R_z\psi\|_{H^1}\le \max\{1,z^{-1}\}\|\psi\|_{L^2}$ for any $\psi\in L^2(\Omega)$. Since the embedding of $H^1(\Omega)$ in $L^2(\Omega)$ is compact, the operator $R_z : L^2(\Omega) \to L^2(\Omega)$ is compact.
  \end{proof}